\documentclass{amsart}

\usepackage{amsfonts}
\usepackage{amssymb}
\usepackage{amscd}
\usepackage{pictexwd,dcpic}
\usepackage{graphicx}
\usepackage{hyperref}
\hypersetup{
    colorlinks=true,
    linkcolor=black,
    filecolor=magenta, 
    urlcolor=cyan,
    citecolor=black,}
\usepackage[nameinlink,capitalize]{cleveref}
\usepackage{mathrsfs}
\usepackage{enumerate}

\newtheorem{theorem}{Theorem}[section]

\newtheorem{lemma}[theorem]{Lemma}
\newtheorem{algorithm}[theorem]{Algorithm}

\newtheorem{corollary}[theorem]{Corollary}
\newtheorem{proposition}[theorem]{Proposition}

\theoremstyle{definition}

\newtheorem{setting}[theorem]{Setting}
\newtheorem{notation}[theorem]{Notation}

\newtheorem{remark}[theorem]{Remark}
\newtheorem{definition}[theorem]{Definition}
\numberwithin{equation}{section}
\newtheorem{example}[theorem]{Example}

\def\coker{\mathop{\rm coker}}

\def\sym{\mathop{\rm Sym}}
\def\fitt{\mathop{\rm Fitt}}

\def\rk{\mathop{\rm rank}}

\def\spec{\mathop{\rm Spec}}
\def\min{\mathop{\rm min}}

\def\pf{\mathop{\rm Pf}}
\def\hgt{\mathop{\rm ht}}
\def\dim{\mathop{\rm dim}}

\def\gcd{\mathop{\rm gcd}}

\def\dep{\mathop{\rm depth}}

\def\reg{\mathop{\rm reg}}

\def\p{\mathfrak{p}}

\def\a{\mathfrak{a}}
\def\m{\mathfrak{m}}
\def\n{\mathfrak{n}}
\def\tt{\mathfrak{t}}

\def\A{\mathcal{A}}
\def\D{\mathcal{D}}
\def\DD{\mathscr{D}}
\def\LL{\mathscr{L}}

\def\rr{\mathcal{R}}
\def\ff{\mathcal{F}}
\def\gr{\mathcal{G}}

\def\Q{\mathcal{Q}}

\def\det{\operatorname{det}}

\def\x{\underline{x}}

\def\a{\underline{a}}
\def\T{\underline{T}}

\def\fitt{{\rm Fitt}}
\def\rt{{\rm rt}}
\def\J{\mathcal{J}}
\def\H{\mathcal{H}}
\def\K{\mathcal{K}}

\def\BB{\mathcal{B}}  
\def\BL{\mathcal{L}}

\crefname{lemma}{lemma}{lemmas}
\Crefname{lemma}{Lemma}{Lemmas}
\crefname{observation}{observation}{observations}
\Crefname{observation}{Observation}{Observations}

\crefname{algorithm}{algorithm}{algorithms}
\Crefname{algorithm}{Algorithm}{Algorithms}

\crefname{setting}{setting}{settings}
\Crefname{setting}{Setting}{Settings}
\crefname{theorem}{theorem}{theorems}
\Crefname{theorem}{Theorem}{Theorems}
\crefname{proposition}{proposition}{propositions}
\Crefname{proposition}{Proposition}{Propositions}
\crefname{remark}{remark}{remarks}
\Crefname{Remark}{Remark}{Remarks}
\crefname{corollary}{corollary}{corollaries}
\Crefname{Corollary}{Corollary}{Corollaries}
\crefname{definition}{definition}{definitions}
\Crefname{definition}{Definition}{Definitions}
\crefname{notation}{notation}{notations}
\Crefname{notation}{Notation}{Notations}
\crefname{section}{section}{sections}
\Crefname{section}{Section}{Sections}

\title[The Equations of Rees Algebras of Height Three Gorenstein Ideals]{The Equations of Rees Algebras of Height Three Gorenstein Ideals in Hypersurface Rings}

\author{Matthew Weaver}
\address{Department of Mathematics, University of Notre Dame, Notre Dame IN 46556, USA}
\email{mweaver6@nd.edu}
\date{}

\begin{document}
\maketitle

\begin{abstract}
We study the Rees algebra of a perfect Gorenstein ideal of codimension 3 in a hypersurface ring. We provide a minimal generating set of the defining ideal of these rings by introducing a modified Jacobian dual and applying a recursive algorithm. Once the defining equations are known, we explore properties of these Rees algebras such as Cohen-Macaulayness and Castelnuovo-Mumford regularity.
\end{abstract}

\section{Introduction}\label{intro}

In this paper we consider the Rees algebra of a particular class of ideals and explore the properties of these rings. For $I =(f_1,\ldots,f_n)$ an ideal of a Noetherian ring $R$, the Rees algebra of $I$ is the graded subalgebra $\rr(I) = R[f_1 t,\ldots, f_n t] = R\oplus It \oplus I^2t^2\oplus \cdots \subset R[t]$. Geometrically, $\rr(I)$ is the homogeneous coordinate ring of the blowup of $\spec(R)$ along the closed subscheme $V(I)$. There is a natural $R$-algebra epimorphism $\Psi:R[T_1,\ldots,T_n] \rightarrow \rr(I)$ given by $\Psi(T_i)= f_it$. The kernel $\J = \ker \Psi$ is the \textit{defining ideal} of $\rr(I)$ and is of great interest as $\Psi$ induces an isomorphism $\rr(I) \cong R[T_1,\ldots,T_n] / \J$. The search for a set of minimal generators of $\J$, the \textit{defining equations} of $\rr(I)$, has become a fundamental problem and has been studied to great extent in recent years (see e.g. \cite{MU,KPU2,Morey,UV,KM,Vasconcelos,HSV1,Johnson,CHW,KPU1,BCS,Weaver}).

Although this problem has been well-studied, the defining equations of Rees algebras are known in few cases. As $\J$ encodes all of the polynomial relations amongst a generating set of $I$, a complete solution requires some knowledge regarding the structure of $I$ and its syzygies. Much work has been accomplished for perfect ideals of grade two, which are generated by the maximal minors of an almost square matrix by the Hilbert-Burch theorem \cite{Eisenbud}. These ideals and their Rees algebras have been studied under a multitude of various assumptions (see e.g. \cite{Morey, MU,BM,CHW,KM,Lan1,Lan2,Weaver}). Furthermore, perfect Gorenstein ideals of grade three and their Rees algebras have been a topic of great interest in recent years. Similar to perfect ideals of grade two, these ideals have prescribed structures and resolutions. These ideals are generated by the submaximal Pfaffians of a square alternating matrix by the Buchsbaum-Eisenbud structure theorem \cite{BE}. The Rees algebras of these ideals have been studied in different settings using a variety of techniques (see e.g. \cite{Johnson,Morey,KPU2}). 


In this paper, we consider the Rees algebra of a perfect Gorenstein ideal of grade three in a \textit{hypersurface ring}. Whereas the defining equations of Rees rings have been studied to great length, most results within the literature require that the ideal in question belongs to $k[x_1,\ldots,x_d]$, a polynomial ring over a field $k$. There is strong geometric motivation to consider Rees algebras of ideals in this new setting. As the Rees ring is the algebraic realization of the blowup of $\spec(R)$ along $V(I)$, altering the ring is reflected by the blowup of a different scheme. There has been recent success in the way of determining the equations defining Rees algebras of perfect ideals with grade two in hypersurface rings in \cite{Weaver}. Expanding upon this, we consider perfect Gorenstein ideals of grade three in these rings and study the defining equations of their Rees algebras.

The objective of this paper is to extend one of the classical results within the study of Rees algebras to the setting of a hypersurface ring. In \cite{Morey}, Morey considered a linearly presented perfect Gorenstein ideal of grade three in $k[x_1,\ldots,x_d]$. The defining equations of its Rees ring were produced and it was shown that there is a single nontrivial equation, which can be identified as the greatest common divisor of the maximal minors of a Jacobian dual matrix. In our setting, we show that this fails to be the case, but that a similar phenomenon occurs upon modification and repetition. The main results \Cref{mainresult} and \Cref{depth}(a) are rephrased below.

\begin{theorem}
Let $S=k[x_1,\ldots,x_{d+1}]$ for $k$ an infinite field, $f\in S$ a homogeneous polynomial of degree $m$, and $R=S/(f)$. Let $I$ be a perfect Gorenstein $R$-ideal of grade 3 with alternating presentation matrix $\varphi$ consisting of linear entries. Let $\overline{\,\cdot\,}$ denote images modulo $(f)$. If $I$ satisfies $G_d$, $I_1(\varphi) = \overline{(x_1,\ldots,x_{d+1})}$, and $\mu(I) = d+1$, then the defining ideal $\J$ of $\rr(I)$ is 
$$\J =\overline{\BL_m + \big(\gcd I_{d+1}(\BB_m)\big)}$$
where the pair $(\BB_m,\BL_m)$ is the $m^{\text{th}}$ gcd-iteration of $(B,\LL)$, for $B$ a modified Jacobian dual with respect to $\x=x_1,\ldots,x_{d+1}$ and $\LL=(\x\cdot B)$. Additionally, $\rr(I)$ is almost Cohen-Macaulay and is Cohen-Macaulay if and only if $m=1$.
\end{theorem}

Traditionally, one searches for the nontrivial equations of Rees algebras by using a \textit{Jacobian dual} matrix corresponding to a presentation matrix of the ideal. However, in the setting above, the Jacobian dual is insufficient and such a matrix must be altered. Repeating the construction in \cite{Weaver}, we introduce a \textit{modified Jacobian dual} matrix. A recursive algorithm of \textit{gcd-iterations} is then developed in order to produce the equations of $\J$. This iterative procedure is similar to the methods used in \cite{BM,CHW,Weaver}.

We now describe how this paper is organized. In \Cref{prelims} we briefly review the preliminary material on Rees algebras of ideals necessary for the scope of this paper. Additionally, we restate the result of Morey \cite[4.3]{Morey} and describe some properties of the Jacobian dual of an alternating matrix. In \Cref{hypring} we begin the study of the Rees algebra $\rr(I)$, for $I$ a linearly presented perfect Gorenstein ideal of grade three in a hypersurface ring $R=S/(f)$. We introduce a perfect Gorenstein ideal $J$ of grade three in the polynomial ring $S$ and compare the Rees algebras $\rr(I)$ and $\rr(J)$. We also introduce the modified Jacobian dual matrix. In \Cref{iterationssec} we introduce the recursive algorithm of \textit{gcd-iterations}, which produces equations belonging to the defining ideal. We then give a sufficient condition for when the defining ideal agrees with the ideal obtained from this algorithm. In \Cref{defidealsec} we show that this condition is satisfied and that the method of gcd-iterations produces a \textit{minimal} generating set of $\J$. Properties such as Cohen-Macaulayness and Castelnuovo-Mumford regularity of $\rr(I)$ are then studied.

\section{ Preliminaries}\label{prelims}

We now introduce the necessary conventions and preliminary information required for this paper. 

\subsection{Rees Algebras of Ideals}

 Let $R$ be a Noetherian ring and $I=(f_1, \ldots,f_n)$ an $R$-ideal of positive grade. There is a natural homogeneous epimorphism of $R$-algebras
$$\Psi:\, R[T_1,\ldots,T_n] \longrightarrow \rr(I)$$
given by $T_i\mapsto f_it$. This map induces an isomorphism
$$\rr(I) \cong R[T_1,\ldots,T_n]/\J$$
for $\J = \ker \Psi$, which is the \textit{defining ideal} of $\rr(I)$. Additionally, any minimal generator of $\J$ is called a \textit{defining equation} of $\rr(I)$. The map $\Psi$ factors through the symmetric algebra $\sym(I)$ via the natural map
$$\sigma:\, R[T_1,\ldots,T_n] \longrightarrow \sym(I)$$
where the kernel $\LL=\ker \sigma$ can be described easily from a presentation of $I$. Indeed, if $R^m\overset{\varphi}{\rightarrow}R^n \rightarrow I\rightarrow 0$ is any presentation of $I$, then $\LL$ is generated by the linear forms $\ell_1,\ldots,\ell_m$ where
$$[T_1\ldots T_n]\cdot \varphi = [\ell_1 \ldots \ell_m].$$

We note that if $R$ is a standard graded ring, each of the maps and ideals above are bihomogeneous. As $\Psi$ factors through $\sigma$, we have the containment of ideals $\LL\subseteq \J$. This containment is often strict, but if $\LL = \J$ we say that $I$ is of \textit{linear type}. As mentioned, $\Psi$ factors through $\sigma$, hence there is a natural epimorphism $\sym(I) \rightarrow \rr(I)$ with kernel $\Q =\J/\LL$. This ideal $\Q$ is typically used to measure how greatly $\sym(I)$ and $\rr(I)$ differ when $I$ is not an ideal of linear type.

We now introduce a common source of higher-degree generators of $\J$. With $R^m\overset{\varphi}{\rightarrow}R^n \rightarrow I\rightarrow 0$ a presentation of $I$ as before, there exists an $r\times m$ matrix $B(\varphi)$ consisting of linear entries in $R[T_1,\ldots,T_n]$ with
$$[T_1 \ldots T_n] \cdot \varphi= [x_1\ldots x_r]\cdot B(\varphi) $$
where $(x_1,\ldots,x_r)$ is an ideal containing the entries of $\varphi$. The matrix $B(\varphi)$ is called a \textit{Jacobian dual} of $\varphi$, with respect to the sequence $x_1,\ldots,x_r$. Notice that $[x_1\ldots x_r]\cdot B(\varphi) = [\ell_1\ldots \ell_m]$, where $\ell_1,\ldots,\ell_m$ are the equations defining $\sym(I)$ as before. We note that this matrix is not unique in general. However, if $R=k[x_1,\ldots,x_d]$ and the entries of $\varphi$ are linear, there is a Jacobian dual $B(\varphi)$, with respect to $x_1,\ldots,x_d$, consisting of linear entries in $k[T_1,\ldots,T_n]$ which is unique.

The ideal $I$ is said to satisfy the condition $G_s$ if $\mu(I_\p) \leq \dim R_\p$ for all $\p \in V(I)$ with $\dim R_\p\leq s-1$. Equivalently, $I$ satisfies $G_s$ if and only if $\hgt \fitt_j(I)\geq j+1$ for all $1\leq j\leq s-1$, where $\fitt_j(I) = I_{n-j}(\varphi)$ is the $j^{\text{th}}$ \textit{Fitting ideal} of $I$ for any presentation $\varphi$ as above. If $I$ satisfies $G_s$ for all $s$, $I$ is said to satisfy $G_\infty$.

Lastly, we introduce two algebras related to $\rr(I)$. The \textit{associated graded ring} of $I$ is $\gr(I) =\rr(I) \otimes_R R/I\cong  \rr(I)/I\rr(I)$. If $R$ is a local ring with maximal ideal $\m$ and residue field $k$, the \textit{special fiber ring} of $I$ is $\ff(I) = \rr(I)\otimes_R k \cong \rr(I)/\m \rr(I)$. Its Krull dimension $\ell(I) = \dim \ff(I)$ is the \textit{analytic spread} of $I$.

\subsection{Perfect Gorenstein Ideals of Grade Three}
The Rees algebras of perfect Gorenstein ideals of grade three are a rich source of interesting phenomena and anomalies. These ideals are a natural candidate to study due to their prescribed structures and resolutions. By the Buchsbaum-Eisenbud theorem \cite[2.1]{BE}, these ideals are presented by an alternating matrix. Moreover, these ideals can be generated by the submaximal Pfaffians of such a matrix. We begin by recalling a classic result within the study of equations of Rees algebras due to Morey and we restate it for reference.

\begin{theorem}[{\cite[4.3]{Morey}}]\label{Moreyresult}
Let $R=k[x_1,\ldots,x_d]$ for $k$ a field, and let $I$ be a perfect Gorenstein $R$-ideal of grade 3 with alternating presentation matrix $\varphi$ consisting of linear entries. If $I$ satisfies the condition $G_d$ and $\mu(I)=d+1$, then the defining ideal of $\rr(I)$ is 
$$\J = \LL + \big(\gcd I_d(B(\varphi))\big)$$
where $B(\varphi)$ is the Jacobian dual of $\varphi$ with respect to $\x=x_1,\ldots,x_d$ and $\LL = (\x\cdot B(\varphi))$. Additionally, $\rr(I)$ is Cohen-Macaulay.
\end{theorem}

 Here $\gcd I_d(B(\varphi))$ denotes the greatest common divisor of the $d\times d$ minors of the Jacobian dual $B(\varphi)$, which consists of entries in $k[T_1,\ldots,T_{d+1}]$. Whereas this result does describe a generating set of the defining ideal, for our purposes we will require a more precise description in how this greatest common divisor arises. In order to describe how the minors of the Jacobian dual above factor, we introduce a lemma of Cramer's rule.

\begin{lemma}[{\cite[4.3]{BM}}]\label{crlemma}
Let $R$ be a commutative ring, $[a_1\ldots a_r]$ a $1\times r$ matrix, and $M$ an $r\times (r-1)$ matrix with entries in $R$. For $1\leq t\leq r$, let $M_t$ denote the $(r-1)\times (r-1)$ submatrix of $M$ obtained by deleting the $t^{\text{th}}$ row of $M$ and set $m_t=\det M_t$. Then in the ring $R/(\a\cdot M)$
$$\overline{a_t}\cdot\overline{m_k} = (-1)^{t-k} \overline{a_k}\cdot \overline{m_t}$$
for all $1\leq k,t\leq r$.
\end{lemma}

With this, we may describe how the maximal minors of $B(\varphi)$ factor in the setting of \Cref{Moreyresult}.

\begin{proposition}\label{JDminors}
With the assumptions of \cref{Moreyresult}, let $B_i$ denote the submatrix obtained by deleting the $i^{\text{th}}$ column of $B(\varphi)$. There exists a polynomial $g \in k[T_1,\ldots, T_{d+1}]$ such that for all $1\leq j \leq d+1$, one has $ \det B_i = (-1)^{i+1} T_i \cdot g$. 
\end{proposition}

\begin{proof}
Writing $\x=x_1,\ldots,x_d$ and $\T =T_1,\ldots,T_{d+1}$ for the two sequences of indeterminates, we claim that $B(\varphi) \cdot [\,\T\,]^t = 0$. As $[\,\x\,] \cdot B(\varphi) = [\,\T\,] \cdot \varphi$, we have 
$$[\,\x\,] \cdot B(\varphi) \cdot [\,\T\,]^t = [\,\T\,] \cdot \varphi \cdot [\,\T\,]^t =0 $$
since $\varphi$ is an alternating matrix. As $\varphi$ consists of linear entries in $k[x_1,\ldots,x_d]$, the entries of $B(\varphi)$ belong to $k[T_1,\ldots,T_{d+1}]$. Thus it follows that $B(\varphi) \cdot [\,\T\,]^t = 0$. Now applying \cref{crlemma} to $[\,\T\,]$ and the transpose of $B(\varphi)$, it follows that 
$$T_i \cdot (\det B_j) = (-1)^{i-j} \,T_j \cdot (\det B_i)$$
in $k[T_1,\ldots,T_{d+1}]$ for all $1\leq i,j\leq d+1$, and the claim follows.
\end{proof}

Notice that the equation in \Cref{JDminors} is precisely the greatest common divisor of the minors of $B(\varphi)$, $g=\gcd I_d(B(\varphi))$, as in \Cref{Moreyresult}.

\section{Ideals of Hypersurface Rings}\label{hypring}

We now begin our study of the Rees algebra $\rr(I)$, for $I$ a perfect Gorenstein ideal of grade three in a hypersurface ring. We introduce a second ideal $J$ which is also perfect Gorenstein of grade three and is closely related to $I$. We study the relation between the Rees rings $\rr(I)$ and $\rr(J)$, and their defining ideals.

\begin{setting}\label{setting1}
Let $S=k[x_1,\ldots,x_{d+1}]$ for $k$ an infinite field, $f\in S$ a homogeneous polynomial of degree $m\geq 1$, and $R= S/(f)$. Let $I$ be a perfect Gorenstein $R$-ideal of grade 3 with alternating presentation matrix $\varphi$ consisting of linear entries in $R$. Further assume that $I$ satisfies the condition $G_d$, $\mu(I)=d+1$, and $I_1(\varphi) = \overline{(x_1,\ldots,x_{d+1})}$.
\end{setting}

Notice that $d$ is necessarily even by \cite[2.2]{BE}. Following the path of \cite{Weaver}, we immediately return to the polynomial ring and produce an $S$-ideal related to $I$, which will also be perfect and Gorenstein of grade three.

\begin{notation}\label{notation1}
Let $\overline{\,\cdot\,}$ denote images  modulo the ideal $(f)$ and let $\psi$ be an $(d+1)\times (d+1)$ alternating matrix consisting of linear entries in $S$ with $I_1(\psi) =(x_1,\ldots,x_{d+1})$, such that $\varphi = \overline{\psi}$. Writing $[\ell_1 \ldots \ell_{d+1}]= [T_1 \ldots T_{d+1}]  \cdot \psi$, we consider the $S[T_1,\ldots,T_{d+1}]$-ideal $\LL=(\ell_1,\ldots,\ell_{d+1},f)$.  
\end{notation}


Certainly, such a matrix $\psi$ exists and we note that it is unique and automatically has $I_1(\psi) =(x_1,\ldots,x_{d+1})$ if $m\geq 2$. If $m=1$, then $\psi$ is not unique, but any such matrix may be chosen.

\begin{proposition}\label{Jlineartype}
There exists a perfect Gorenstein $S$-ideal $J$ with grade 3, which is presented by $\psi$. Additionally, $J$ is of linear type.
\end{proposition}

\begin{proof}
To show that $\psi$ is the presentation matrix of a perfect Gorenstein ideal with grade 3, it suffices to show that $\hgt \pf_d(\psi) \geq 3$ by \cite[2.1]{BE}. Notice that the image of this ideal in $R$ is exactly the corresponding ideal of Pfaffians of $\varphi$. As the height can only decrease by passing to $R$, we have $\hgt \pf_d(\psi) \geq \hgt \overline{\pf_d(\psi)} = \hgt \pf_d(\varphi)=3$, as $I$ is perfect and Gorenstein of grade 3, using \cite[2.1]{BE}. Thus the first claim follows and such an ideal $J$ exists.

To show that $J$ is of linear type, it suffices to show that $J$ satisfies $G_{\infty}$ by \cite[2.6]{HSV1}. However, as $\mu(J) = d+1 =\dim S$, it is enough to show that $J$ satisfies $G_{d+1}$. Recall from \Cref{prelims} that this condition can be interpreted in terms of heights of Fitting ideals. Repeating the previous argument, notice that the images of the Fitting ideals of $J$ in $R$ are the corresponding Fitting ideals of $I$. Moreover, the heights of these ideals can only decrease when passing to $R$. Hence $\hgt \fitt_i(J) \geq \hgt \fitt_i(I) \geq i+1$ for all $1\leq i\leq d-1$, since $I$ satisfies $G_d$. With this, it follows that $J$ satisfies $G_d$ as well. Thus we need only show that the $d^{\text{th}}$ Fitting ideal of $J$ has height at least, and hence equal to, $d+1$ to conclude that $J$ satisfies $G_{d+1}$. However, this ideal is $\fitt_d(J) = I_1(\psi) = (x_1,\ldots,x_{d+1})$, which of course has maximal height.
\end{proof}

As $J$ is of linear type, notice that the $S[T_1,\ldots,T_{d+1}]$-ideal $(\ell_1,\ldots,\ell_{d+1})$ is precisely the ideal defining $\sym(J) \cong \rr(J)$. Moreover, it follows that $\overline{\LL}$ is the defining ideal of $\sym(I)$, as $\varphi = \overline{\psi}$. With this, we see that $S[T_1,\ldots,T_{d+1}]/ \LL \cong R[T_1,\ldots,T_{d+1}]/\overline{\LL} \cong \sym(I)$, hence $\LL$ is the  ideal defining $\sym(I)$, as a quotient of $S[T_1,\ldots,T_{d+1}]$. Hence there is a clear relation between the $S[T_1,\ldots,T_{d+1}]$-ideals defining $\sym(J)$ and $\sym(I)$, as these ideals differ only by the generator $f$. Naturally one could ask if there is a similar connection between the ideals defining $\rr(J)$ and $\rr(I)$. Before we answer this, we provide an alternative description of the defining ideal $\J$ of $\rr(I)$ and then introduce an ideal defining $\rr(I)$ as a quotient of $S[T_1,\ldots,T_{d+1}]$.

\begin{proposition}\label{Jasat}
With the assumptions of \Cref{setting1} and $\LL$ as in \Cref{notation1}, the defining ideal of $\rr(I)$ satisfies $\J = \overline{\LL:(x_1,\ldots,x_{d+1})^\infty}$.
\end{proposition}

\begin{proof}
As $I$ satisfies the condition $G_d$, for any non-maximal homogeneous prime $R$-ideal $\p$, $I_\p$ satisfies $G_\infty$ as an $R_\p$-ideal and is hence of linear type by \cite[2.6]{HSV1}. Thus $\J_\p = \overline{\LL}_\p$ for any such prime ideal $\p$ and so the quotient $\Q = \J/\overline{\LL}$ is supported only at the homogeneous maximal $R$-ideal $\overline{(x_1,\ldots,x_{d+1})}$. Hence $\Q$ is annihilated by some power of $\overline{(x_1,\ldots,x_{d+1})}$, which shows that $\J \subseteq \overline{\LL}:\overline{(x_1,\ldots,x_{d+1})}^\infty$. However, we have the containment $\overline{\LL}:\overline{(x_1,\ldots,x_{d+1})}^\infty \subseteq \J$ as $\overline{\LL} \subseteq \J$ and modulo $\J$, the image of $\overline{(x_1,\ldots,x_{d+1})}$ in $\rr(I)$ is an ideal of positive grade.
\end{proof}

The statement regarding the grade of $\overline{(x_1,\ldots,x_{d+1})}\rr(I)$ in the proof above follows from the well-known correspondence between the associated primes of $R$ and $\rr(I)$ \cite{HS}\cite[1.5]{EHU}.

Notice that $\Q=\J/\overline{\LL}$, as in the proof of \Cref{Jasat}, is the kernel of the natural bihomogeneous map $\sym(I(\delta))\rightarrow \rr(I)$, where $\delta = \frac{d}{2}$ is the degree of the generators of $I$ (recall that $d$ is even). Writing $\m = (x_1,\ldots,x_{d+1})$ for the homogeneous maximal $S$-ideal, the description of $\J$ in \Cref{Jasat} shows that $\Q\cong H_{\overline{\m}}^0\big(\sym(I(\delta))\big)$, the zeroth local cohomology module of $\sym(I(\delta))$ with support in $\overline{\m}$. Thus $\Q$ is concentrated in only finitely many degrees and so we may use the tools developed in \cite{KPU3} to bound these degrees.

\begin{proposition}\label{IndexOfSat}
With the assumptions of \Cref{setting1}, $\J=\overline{\LL:(x_1,\ldots,x_{d+1})^m}$.
\end{proposition}

\begin{proof}
It is clear that $\overline{\LL:(x_1,\ldots,x_{d+1})^m} \subseteq \J$, following \Cref{Jasat}, hence we need only show the reverse containment. In order to show that $\overline{\m}^m \J \subseteq \overline{\LL}$, it suffices to show that $\overline{\m}^m \Q =0$, where $\Q$ is as above. As mentioned, $\Q\cong H_{\overline{\m}}^0\big(\sym(I(\delta))\big)$ and so, with the bigrading $\deg \overline{x_i}=(1,0)$ and $\deg T_i = (0,1)$ on $R[T_1,\ldots,T_{d+1}]$, we may write
$$\Q_{(*,q)} = \bigoplus_p \Q_{(p,q)} \cong  H_{\overline{\m}}^0\big({\rm Sym}_q(I(\delta)) \big).$$
As $\Q$ lives in finitely many degrees, it is enough to show that $\Q$ vanishes past degree $m-1$ in the first component of the bigrading. 

By \cite[3.8]{KPU3}, it follows that $\Q_{(p,q)} =0$ for all $p > b_0 (\DD_d^q) +a(R)$ and any $q$, where $\DD_d^q$ is the $d^{\text{th}}$ module of a homogeneous complex $\DD_{\bullet}^q$ of finitely generated graded $R$-modules with zeroth homology $H_0(\DD_{\bullet}^q) \cong \sym_q(I(\delta))$, $b_0 (\DD_d^q)$ is the \textit{maximal generator degree} of $\DD_d^q$ from \cite[2.2]{KPU3}, and $a(R)$ is the $a$-invariant of $R$. Since $R$ is a Cohen-Macaulay $k$-algebra, the $a$-invariant of $R$ is $a(R) = \reg R -d$, where $\reg R$ denotes the Castelnuovo-Mumford regularity of $R$. As $R$ is a hypersurface ring defined by a polynomial of degree $m$, it follows that $a(R) = \reg R -d = (m-1)-d$, hence we need only show that $b_0 (\DD_d^q) \leq d$ for any $q$.

Since $\varphi$ is a $(d+1)\times (d+1)$ homogeneous alternating matrix which presents $I$ minimally, we may take
$$\DD_{\bullet}^q (\varphi): \, 0 \longrightarrow \DD_d^q \longrightarrow \DD_{d-1}^q \longrightarrow \cdots\cdots \longrightarrow \DD_1^q \longrightarrow \DD_0^q\longrightarrow 0$$
to be the complex from \cite[2.15, 4.7]{KU} associated to $\varphi$. The zeroth homology of $\DD_{\bullet}^q (\varphi)$ is $H_0\big(  \DD_{\bullet}^q (\varphi) \big) \cong \sym_q(I(\delta))$, hence we may consider this complex and the maximal generator degree of $\DD_d^q$. Following the description and notation of this complex given in \cite{KU} and restated in \cite{KPU2}, and noting that the entries of $\varphi$ are linear, for all $1\leq r\leq d$ we have
\[
\DD_r^q = \left\{
     \begin{array}{ll}
       K_{q-r,r} = R(-r)^{\beta_r^q} & \text{if $r \leq \min\{q,d\}$}\\[1ex]
       Q_q = R(-(r-1)-\frac{1}{2}(d-r+2)) & \text{if $r =q+1\leq d$, $q$ odd}\\[1ex]
       0& \text{if $r =q+1$, $q$ even}\\[1ex]
       0& \text{if $r \geq \min\{q+2,d+1\}$}\\[1ex]
     \end{array}
   \right.
\]
for some nonzero Betti numbers $\beta_r^q$. Allowing $r=d$, which is even, the expressions above simplify to 
\[
\DD_d^q = \left\{
     \begin{array}{ll}
       K_{q-d,d} = R(-d)^{\beta_r^q} & \text{if $q\geq d$}\\[1ex]
       Q_q = R(-d) & \text{if $q =d-1$}\\[1ex]
       0& \text{if $q\leq d-2$}.\\[1ex]
     \end{array}
   \right.
\]
If $\DD_d^q=0$, then $b_0(\DD_d^q) = -\infty$ by convention Moreover, if $\DD_d^q\neq 0$, we see that $b_0(\DD_d^q) =d$. Thus $b_0 (\DD_d^q) \leq d$ for any $q$, as required.
\end{proof}

\begin{notation}\label{Anotation}
With the result of \Cref{IndexOfSat}, let us denote the ideal $\A = \LL:(x_1,\ldots,x_{d+1})^m = \LL:(x_1,\ldots,x_{d+1})^\infty$ in $S[T_1,\ldots,T_{d+1}]$.
\end{notation}

Notice that $\overline{\A} = \J$, hence $\A$ defines $\rr(I)$ as a quotient of $S[T_1,\ldots,T_{d+1}]$ since $S[T_1,\ldots,T_{d+1}]/\A \cong R[T_1,\ldots,T_{d+1}]/\J \cong \rr(I)$. For much of the duration of this paper, the ideal $\A$ will be the object of our focus. This ideal is a defining ideal of $\rr(I)$, in a sense, and belongs to the polynomial ring $S[T_1,\ldots,T_{d+1}]$, where a colon ideal $\A = \LL:(x_1,\ldots,x_{d+1})^m$ is more easily studied.

We follow a path parallel to the traditional one by approximating the defining ideal of $\rr(I)$ using the defining ideal of $\sym(I)$, now using the ideals $\A$ and $\LL$ in place of $\J$ and $\overline{\LL}$. Traditionally, one then employs a Jacobian dual matrix, however as we have updated our ideals, we must also update such a matrix. We recall the notion of a \textit{modified Jacobian dual} as presented in \cite{Weaver}. This matrix is associated to the generators of $\LL$ and the sequence $x_1,\ldots,x_{d+1}$. 

Before introducing this object, we provide the notation necessary for its definition and the constructions in the proceeding section. Notice that $S[T_1,\ldots,T_{d+1}]$ is naturally bigraded with $\deg x_i = (1,0)$ and $\deg T_i = (0,1)$.

\begin{notation}\label{delnotation}
 For $F \in (x_1,\ldots,x_{d+1})S[T_1,\ldots,T_{d+1}]$ a nonzero bihomogeneous polynomial, let $\partial F$ denote a column consisting of bihomogeneous entries with bidegree $\deg F -(1,0)$, such that $[x_1,\ldots,x_{d+1}]\cdot \partial F = F$. As a convention, we take $\partial F$ to consist of zeros if $F=0$.
\end{notation}

In general, there are many choices for $\partial F$. As the notation suggests, there is a natural choice for $\partial F$ using differentials, if $k$ is a field of characteristic zero. Writing $\deg F = (r,*)$ for $r>0$, we have the Euler formula $r\cdot F \,= \,\sum_{i=1}^{d+1}\,\frac{\partial F}{\partial x_i} \cdot x_i$, noting that $r$ is a unit. Hence $\partial F$ can be chosen to have entries $\frac{1}{r}\cdot\frac{\partial F}{\partial x_i}$ in this setting.

\begin{definition}\label{mjddefn}
With $\LL$ and $\psi$ as in \Cref{notation1}, we take a \textit{modified Jacobian dual} of $\psi$ to be the $(d+1) \times (d+2)$ matrix $B=[B(\psi)\,|\,\partial f]$ where $B(\psi)$ is the Jacobian dual of $\psi$, consisting of linear entries in $k[T_1,\ldots,T_{d+1}]$, and $\partial f$ is a column corresponding to $f$, as in \Cref{delnotation}. Here $|$ denotes the usual matrix concatenation. 
\end{definition}

Notice that the entries of the matrix product $[x_1\ldots x_{d+1}]\cdot B$ are precisely the generators of $\LL$. In the next section we will employ the modified Jacobian dual and similar constructions to produce equations in $\A$. For now however, we must produce another description of $\A$.

Following the approach in \cite{KPU1}, we find a ring which maps onto $\rr(I)$ such that the kernel is an ideal of height one. We take this ring to be the Rees algebra $\rr(J)$, noting that $J$ is of linear type by \Cref{Jlineartype}. We now study how these Rees algebras, and their defining ideals, relate to each other.

\subsection{Ideals in $\rr(J)$}

Before we study the relation between $\rr(J)$ and $\rr(I)$, we introduce a third, and final, perfect Gorenstein ideal of grade three. It will be seen that this ideal satisfies the assumptions of \Cref{Moreyresult}. The defining ideal of its Rees algebra will then be used to produce a description of the image of $\A$ in $\rr(J)$. We begin by providing a short lemma, commonly used to avoid certain ideals in graded rings.

\begin{lemma}\label{idealavoidance}
Let $A=k[y_1,\ldots,y_n]$ for $k$ an infinite field and let $J$ be an ideal generated by homogeneous elements of degree $r$. Suppose that $I_1,\ldots,I_s$ are ideals of $A$, none of which contains $J$. There exists a homogeneous element $z\in J$ of degree $r$ such that $z\notin I_j$ for all $1\leq j\leq s$.
\end{lemma}

\begin{proof}
This follows from the well-known fact that a vector space over an infinite field is not a finite union of proper subspaces, and then applying Nakayama's lemma in the graded setting.
\end{proof}

\begin{proposition}\label{J'ideal}
With the assumptions of \Cref{setting1}, let $S' =k[x_1,\ldots,x_d] \cong S/(x_{d+1})$ and consider the matrix $\psi' = \psi S'$. After a possible linear change of coordinates, there exists an $S'$-ideal $J'$ that is perfect and Gorenstein of grade $3$, which is presented by $\psi'$. Moreover, $J'$ satisfies $G_d$.
\end{proposition}

\begin{proof}
Notice that $\psi'$ is an alternating $(d+1)\times (d+1)$ matrix with entries in $S'$. By \cite[2.1]{BE}, the existence of such an ideal depends only on the height of an ideal of Pfaffians of $\psi' = \psi S'$. Moreover, the condition $G_d$  depends on the heights of ideals of minors of $\psi'$. Thus it suffices to show that, after making a suitable change of coordinates, $\hgt \pf_d(\psi')\geq 3$ and $\hgt I_j(\psi') \geq d-j+2 $ for all $2\leq j \leq d$. Notice that $\pf_d(\psi') = \pf_d(\psi)S'$ and $I_j(\psi') = I_j(\psi)S'$. Recall that $J$ is presented by $\psi$ and is of linear type by \Cref{Jlineartype}, and hence satisfies $G_\infty$.

Recall that $I$ is an ideal of height 3 in $R$, which is $d$-dimensional. As $d$ is even, it follows that $d\geq 4$, hence $\dim S = d+1 \geq 5$ and so $\hgt \pf_d(\psi)=3 <\dim S$. Now consider the determinantal ideals $I_j(\psi)$ with height at most $d$, for $2\leq j\leq d$. There are finitely many non-maximal minimal primes of $\pf_d(\psi)$ and the ideals $I_j(\psi)$ with non-maximal height, and none of them contains $(x_1,\ldots,x_{d+1})$. Hence there exists a linear form not contained in any of these minimal primes by \Cref{idealavoidance}. After a potential linear change of coordinates, it can be assumed that $x_{d+1}$ is precisely this linear form.

With this, we see that $\hgt \pf_d(\psi') = \pf_d(\psi) =3$ and $\hgt I_j(\psi') = \hgt I_j(\psi)\geq d-j+2$ for all $j$ such that $\hgt I_j(\psi)\leq d$ and $2\leq j\leq d$. For any of the ideals $I_j(\psi)$ with maximal height and $j$ in this range, the height of $I_j(\psi)$ must drop when passing to $S'$. However, if $\hgt I_j(\psi) = d+1$, then $\hgt I_j(\psi') = \hgt I_j(\psi) -1 =d \geq d-j+2$ as $2\leq j\leq d$. 
\end{proof}

\begin{remark}
In \Cref{J'ideal}, a linear change of coordinates was made and we note that the conditions and constructions introduced so far are amenable to such a change. We proceed assuming that such a linear adjustment has been made and the sequence $x_1,\ldots,x_{d+1}$ has been relabelled accordingly.  
\end{remark}

 Notice that the $S'$-ideal $J'$ satisfies the assumptions of \Cref{Moreyresult}, hence the defining equations of $\rr(J')$ are known. In particular, they can be described from the Jacobian dual $B(\psi')$, which is precisely a submatrix of $B(\psi)$. Indeed, the entries of $B(\psi)$ belong to $k[T_1,\ldots,T_{d+1}]$ and the last row of $B(\psi)$ corresponds to $x_{d+1}$. Thus by deleting the last row of $B(\psi)$ we obtain the Jacobian dual of $\psi'$, with respect to $x_1,\ldots,x_d$ in $S'$. Letting $B'$ denote this submatrix of $B(\psi)$ obtained by deleting the last row, we note that there is a nontrivial greatest common divisor among the maximal minors of $B'=B(\psi')$, by \Cref{JDminors}.

\begin{notation}\label{notation2}
Recall that $J$ is of linear type by \Cref{Jlineartype}, hence $\rr(J) \cong S[T_1,\ldots,T_{d+1}]/ \H$ where $\H =(\ell_1,\ldots,\ell_{d+1})$, following \Cref{notation1}. Let $\widetilde{\,\cdot\,}$ denote images modulo $\H$, in $\rr(J)$. As before, let $B'$ be the $d\times (d+1)$ matrix obtained by deleting the last row of $B(\psi)$ and consider the $S[T_1,\ldots,T_{d+1}]$-ideal $\K= (\ell_1,\ldots,\ell_{d+1}) + (\gcd I_d(B'))+(x_{d+1})$.
\end{notation}

As mentioned, we may identify $B'$ with the Jacobian dual $B(\psi')$, for $\psi'$ as in \Cref{J'ideal}. Hence there is a greatest common divisor amongst the maximal minors of $B'$, in $k[T_1,\ldots,T_{d+1}]$.

\begin{proposition}\label{PropertiesOfA}
The ring $\rr(J)$ is a Cohen-Macaulay domain of dimension $d+2$ and the ideals $\widetilde{\K}$ and $(\widetilde{x_1,\ldots,x_{d+1}})$ are Cohen-Macaulay $\rr(J)$-ideals of height 1. Moreover, $(\widetilde{x_1,\ldots,x_{d+1}})$ is a prime ideal.
\end{proposition}

\begin{proof}
The claim that $\rr(J)$ is a domain of dimension $d+2$ follows easily as $S$ is a domain of dimension $d+1$ and $J$ is an ideal of positive height \cite{VasconcelosBook}. Additionally, $J$ is of linear type by \Cref{Jlineartype}, hence $\rr(J)$ is Cohen-Macaulay by \cite[2.6]{HSV1}. Moreover, as $J$ is of linear type, its special fiber ring is $\ff(J) \cong k[T_1,\ldots, T_{d+1}]$, hence $(\widetilde{x_1,\ldots,x_{d+1}})$ is indeed a Cohen-Macaulay prime ideal of height 1.

To see that $\widetilde{\K}$ is a Cohen-Macaulay $\rr(J)$-ideal of height 1, notice that $\K$ can be written as $(\ell_1',\ldots,\ell_{d+1}') +(\gcd I_d(B'))+(x_{d+1})$, where $[\ell_1'\ldots \ell_{d+1}'] = [x_1 \ldots x_d] \cdot B'$. Notice that $(\ell_1',\ldots,\ell_{d+1}') +( \gcd I_d(B'))$ is exactly the defining ideal of $\rr(J')$ following \Cref{J'ideal} and \Cref{Moreyresult}. In particular, this ideal is Cohen-Macaulay with height $d$. As $x_{d+1}$ is regular modulo this ideal, it then follows that $\K$ has height $d+1$ and is Cohen-Macaulay. Thus $\widetilde{\K}$ is a Cohen-Macaulay ideal of height 1 in $\rr(J)$. 
\end{proof}

\begin{proposition}\label{colons}
With $\K$ as in \Cref{notation2}, for any positive integer $i$ we have the following.
\begin{enumerate}[(a)]
\setlength\itemsep{1em}
    \item $(\widetilde{x_1,\ldots,x_{d+1}})^i = (\widetilde{x_1,\ldots,x_{d+1}})^{(i)}$
    
    \item $(\widetilde{x_1,\ldots,x_{d+1}})^{(i)}= (\widetilde{x_{d+1}}^i) :_{\rr(J)} \widetilde{\K}^{(i)}$
    
    \item $\widetilde{\K}^{(i)} = (\widetilde{x_{d+1}}^i) :_{\rr(J)}  (\widetilde{x_1,\ldots,x_{d+1}})^{(i)}$
\end{enumerate}
\end{proposition}

\begin{proof} We proceed as in the proof of \cite[3.9]{BM}. 

\begin{enumerate}[(a)]
\setlength\itemsep{1em}
    \item Setting the degrees of the $x_i$ to 1 and the degrees of the $T_i$ to 0 temporarily, we see that $\gr\big((\widetilde{x_1,\ldots,x_{d+1}})\big) \cong \rr(J)$, where $\gr\big((\widetilde{x_1,\ldots,x_{d+1}})\big)$ is the associated graded ring of $(\widetilde{x_1,\ldots,x_{d+1}})$. As $\rr(J)$ is a domain, it follows that $(\widetilde{x_1,\ldots,x_{d+1}})^i = (\widetilde{x_1,\ldots,x_{d+1}})^{(i)}$ for all $i$.

    \item We first claim that $(\widetilde{x_1,\ldots,x_{d+1}})\widetilde{\K}\subseteq (\widetilde{x_{d+1}})$. Recall that $\K = (\ell_1',\ldots,\ell_{d+1}') +(\gcd I_d(B'))+(x_{d+1})$, where $[\ell_1'\ldots\ell_{d+1}'] = [x_1 \ldots x_d] \cdot B'$. Thus modulo $\H$, we see that $(\widetilde{\ell_1',\ldots,\ell_{d+1}'}) \subseteq (\widetilde{x_{d+1}})$. Noting that $B'=B(\psi')$, with $\psi'$ as in \Cref{J'ideal}, we may write $I_d(B') = (g')(T_1,\ldots,T_{d+1})$ where $g' = \gcd (I_d(B'))$, by \Cref{JDminors}. By Cramer's rule we have $(x_1,\ldots,x_d) I_d(B') \subseteq (\ell_1',\ldots,\ell_{d+1}')$, hence 
$$(x_1,\ldots,x_d)(g') \subseteq (\ell_1',\ldots,\ell_{d+1}'):(T_1,\ldots,T_{d+1}).$$ 
However, notice that $(\ell_1',\ldots,\ell_{d+1}')$ is the defining ideal of $\sym(J')$ following \Cref{J'ideal}, hence $\hgt (\ell_1',\ldots,\ell_{d+1}') = d$ by \cite[2.1]{Morey}. Modulo $(\ell_1',\ldots,\ell_{d+1}')$, we then see that $(T_1,\ldots,T_{d+1})$ is an ideal of positive grade, which is annihilated by $(x_1,\ldots,x_d)(g')$, hence $(x_1,\ldots,x_d)(g') \subseteq (\ell_1',\ldots,\ell_{d+1}')$. Noting that $(\widetilde{\ell_1',\ldots,\ell_{d+1}'}) \subseteq (\widetilde{x_{d+1}})$, it then follows that $(\widetilde{x_1,\ldots,x_{d+1}})\widetilde{\K}\subseteq (\widetilde{x_{d+1}})$.

With this, we have $(\widetilde{x_1,\ldots,x_{d+1}})^i\widetilde{\K}^i\subseteq (\widetilde{x_{d+1}}^i)$ for any positive integer $i$. Localizing at height one prime ideals of $\rr(J)$, we see that $(\widetilde{x_1,\ldots,x_{d+1}})^{(i)}\widetilde{\K}^{(i)}\subseteq (\widetilde{x_{d+1}}^i)$ and so $(\widetilde{x_1,\ldots,x_{d+1}})^{(i)}\subseteq (\widetilde{x_{d+1}}^i):\widetilde{\K}^{(i)}$. 
Writing $\K=(\ell_1',\ldots,\ell_{d+1}') + (\gcd I_d(B'))+(x_{d+1})$ as before, recall that $(\ell_1',\ldots,\ell_{d+1}') + (\gcd I_d(B'))$ is the defining ideal of $\rr(J')$. Note that $J'$ is not of linear type as $\mu(J') > \dim S'$, hence $\gcd I_d(B')$ is nonzero in $k[T_1,\ldots,T_{d+1}]$ and so $\widetilde{\K}\nsubseteq (\widetilde{x_1,\ldots,x_{d+1}})$. As $(\widetilde{x_1,\ldots,x_{d+1}})$ is the unique associated prime of $(\widetilde{x_1,\ldots,x_{d+1}})^{(i)}$, it follows that $\widetilde{\K}^{(i)}$ and $(\widetilde{x_1,\ldots,x_{d+1}})^{(i)}$ have no associated prime in common. From this it follows that $(\widetilde{x_{d+1}}^i):\widetilde{\K}^{(i)} \subseteq (\widetilde{x_1,\ldots,x_{d+1}})^{(i)}$.

\item  As before, we have $(\widetilde{x_1,\ldots,x_{d+1}})^{(i)}\widetilde{\K}^{(i)}\subseteq (\widetilde{x_{d+1}}^i)$, hence $\widetilde{\K}^{(i)} \subseteq (\widetilde{x_{d+1}}^i) :_{\rr(J)}  (\widetilde{x_1,\ldots,x_{d+1}})^{(i)}$. To show the reverse containment, recall that $(\widetilde{x_1,\ldots,x_{d+1}})$ is not an associated prime of $\widetilde{\K}^{(i)}$. With this and noting that $\widetilde{x_{d+1}}^i\in \widetilde{\K}^{(i)}$, we see that $(\widetilde{x_{d+1}}^i) :_{\rr(J)}  (\widetilde{x_1,\ldots,x_{d+1}})^{(i)} \subseteq \widetilde{\K}^{(i)}$. \qedhere
\end{enumerate}
\end{proof}

With parts (b) and (c) of \Cref{colons}, one says that $(\widetilde{x_1,\ldots,x_{d+1}})^{(i)}$ and $\widetilde{\K}^{(i)}$ are \textit{linked} \cite{Huneke1}.

\begin{corollary}\label{Kscm}
The $\rr(J)$-ideal $\widetilde{\K}$ is generically a complete intersection and is strongly Cohen-Macaulay.
\end{corollary}

\begin{proof}
Recall from \Cref{PropertiesOfA} that $\widetilde{\K}$ is a Cohen-Macaulay ideal of height one. From the proof of \Cref{colons} we had seen that $(\widetilde{x_1,\ldots,x_{d+1}})$ is not an associated prime of $\widetilde{\K}$. Thus if $\p$ is an associated prime of $\widetilde{\K}$, by \Cref{colons} we have $(\widetilde{x_{d+1}})_\p : \widetilde{\K}_\p = (\widetilde{x_1,\ldots,x_{d+1}})_\p = \rr(J)_\p$. Thus we have $\widetilde{\K}_\p \subseteq (\widetilde{x_{d+1}})_\p$ and so $\widetilde{\K}_\p= (\widetilde{x_{d+1}})_\p$ as $x_{d+1} \in \K$, which shows that $\widetilde{\K}$ is generically a complete intersection.

Notice that $\widetilde{\K} = (\widetilde{g'},\widetilde{x_{d+1}})$ where $g'= \gcd I_d(B')$, hence $\widetilde{\K}$ is a Cohen-Macaulay almost complete intersection ideal, following \Cref{PropertiesOfA}. Moreover, we had just seen that $\widetilde{\K}$ is generically a complete intersection, hence it is a strongly Cohen-Macaulay $\rr(J)$-ideal by \cite[2.2]{Huneke2}. 
\end{proof}

We now give an alternative description of the $\rr(J)$-ideal $\widetilde{\A}$. Notice that this is the kernel of the induced map of Rees algebras $\rr(J)\rightarrow \rr(I)$. Consider the fractional ideal $\frac{\widetilde{f}\,\widetilde{\K}^{(m)}}{\widetilde{x_{d+1}}^m}$ and note that this is actually an $\rr(J)$-ideal by \Cref{colons}, as $f\in (x_1,\ldots,x_{d+1})^m$.

\begin{theorem}\label{DandA}
In $\rr(J)$, we have $\widetilde{\A}= \frac{\widetilde{f}\,\widetilde{\K}^{(m)}}{\widetilde{x_{d+1}}^m}$.
\end{theorem}

\begin{proof}
Writing $\D=\frac{\widetilde{f}\,\widetilde{\K}^{(m)}}{\widetilde{x_{d+1}}^m}$, we begin by showing that $\D\subseteq \widetilde{\A}$. Recall that $\rr(J)$ is a domain by \Cref{PropertiesOfA}. Hence for any $a\in (x_1,\ldots,x_{d+1})^{m}$ we have the equality $\frac{\widetilde{f}\,\widetilde{\K}^{(m)}}{\widetilde{x_{d+1}}^m} \cdot \widetilde{a} =  \frac{\widetilde{a}\,\widetilde{\K}^{(m)}}{\widetilde{x_{d+1}}^m} \cdot \widetilde{f}$. Notice that $\frac{\widetilde{a}\,\widetilde{\K}^{(m)}}{\widetilde{x_{d+1}}^m}$ is an $\rr(J)$-ideal by \Cref{colons}, hence it follows that $\D (\widetilde{x_1,\ldots,x_{d+1}})^{m} \subset (\widetilde{f}) = \widetilde{\LL}$. Thus $\D \subseteq \widetilde{\LL}: (\widetilde{x_1,\ldots,x_{d+1}})^{m} = \widetilde{\A}$, following \Cref{Anotation}.

To show this containment is actually an equality, we proceed as in the proof of \cite[3.10]{BM}. Recall that $\rr(J)$ is a Cohen-Macaulay domain and note that $\widetilde{\K}^{(m)}$ is an unmixed ideal of height one. Equivalently, $\widetilde{\K}^{(m)}$ satisfies Serre's condition $S_2$, as an $\rr(J)$-module. Thus $\D$ is also an unmixed $\rr(J)$-ideal of height one since $\D \cong \widetilde{\K}^{(m)}$ and the condition $S_2$ is preserved under isomorphism. As $\D \subseteq \widetilde{\A}$, it suffices to show that these ideals agree locally at the associated primes of $\D$, in order to conclude that $\D = \widetilde{\A}$. As these associated primes have height one, we show that $\D_\p = \widetilde{\A}_\p$ for any prime $\rr(J)$-ideal $\p$ with height one.

Recall from \Cref{PropertiesOfA} that $(\widetilde{x_1,\ldots,x_{d+1}})$ is a prime ideal of height one in $\rr(J)$. If $\p \neq (\widetilde{x_1,\ldots,x_{d+1}})$, we see that $\widetilde{\A}_\p = \widetilde{\LL}_\p : (\widetilde{x_1,\ldots,x_{d+1}})_\p^m = (\widetilde{f})_\p : \rr(J)_\p$, hence $\widetilde{\A}_\p \subseteq (\widetilde{f})_\p$ and so $\widetilde{\A} = (\widetilde{f})_\p$, as $f\in \A$. Additionally, by \Cref{colons} and repeating the argument in the proof of \Cref{Kscm}, it follows that $\widetilde{\K}_\p = (\widetilde{x_{d+1}})_\p$, hence $\D_\p = (f)_\p$ as well.

 Now suppose that $\p =(\widetilde{x_1,\ldots,x_{d+1}})$ and we first note that $\widetilde{\A}\nsubseteq (\widetilde{x_1,\ldots,x_{d+1}})$. Indeed, the analytic spread of $J$ is $\ell(J) = d+1$ since $\ff(J) \cong k[T_1,\ldots,T_{d+1}]$, which we had seen in the proof of \Cref{PropertiesOfA}. Moreover, we have $\ell(I) =d$ by \cite[4.3]{UV}. With the isomorphism $\rr(I) \cong \rr(J) / \widetilde{\A}$ and passing to $\ff(I)$, it then follows that $\widetilde{\A}\nsubseteq (\widetilde{x_1,\ldots,x_{d+1}})$. With this, we see that $\widetilde{\A}_\p = \rr(J)_\p$. 
 
 Recall from the proof of \Cref{colons} that $\widetilde{\K} \nsubseteq (\widetilde{x_1,\ldots,x_{d+1}})$, hence $\widetilde{\K}^{(m)}_\p=\rr(J)_\p$ as well. With this and \Cref{colons}, we see that $(\widetilde{x_1,\ldots,x_{d+1}})_\p =  (\widetilde{x_{d+1}})_\p:\rr(J)_\p$, hence $(\widetilde{x_1,\ldots,x_{d+1}})_\p = (\widetilde{x_{d+1}})_\p$. Thus 
 $$\rr(J)_\p = \widetilde{\A}_\p=\widetilde{\LL}_\p: (\widetilde{x_1,\ldots,x_{d+1}})^m_\p=(\widetilde{f})_\p: (\widetilde{x_1,\ldots,x_{d+1}})^m_\p$$
and so $(\widetilde{x_1,\ldots,x_{d+1}})^m_\p \subseteq (\widetilde{f})_\p$. However, as $f\in(x_1,\ldots,x_{d+1})^m$, we have $(\widetilde{f})_\p = (\widetilde{x_1,\ldots,x_{d+1}})^m_\p =  (\widetilde{x_{d+1}})^m_\p$, hence $\D_\p =\widetilde{\K}^{(m)}_\p=\rr(J)_\p = \widetilde{\A}_\p$.
\end{proof}

Recall that, as a consequence of \Cref{IndexOfSat}, the ideal $\A$ is a saturation that can be written as $\A =\LL:(x_1,\ldots,x_{d+1})^\infty = \LL:(x_1,\ldots,x_{d+1})^m$. We end this section by showing that $m$ is the smallest integer for which this second equality holds, i.e. the index of saturation of $\A$.

\begin{proposition}\label{nissmallest}
With the assumptions of \Cref{setting1}, $m$ is the smallest integer such that $\A=\LL:(x_1,\ldots,x_{d+1})^m$.
\end{proposition}

\begin{proof}
Suppose, for a contradiction, that there is some positive integer $i<m$ such that $\A=\LL:(x_1,\ldots,x_{d+1})^i$. In $\rr(J)$, we then have $\widetilde{\A}=\widetilde{\LL}:(\widetilde{x_1,\ldots,x_{d+1}})^i$. Now localizing at $\p=(\widetilde{x_1,\ldots,x_{d+1}})$ and noting that $\widetilde{\A}_\p = \rr(J)_\p$, as we had seen in the proof of \Cref{DandA}, we have $\rr(J)_\p=(\widetilde{f})_\p:(\widetilde{x_1,\ldots,x_{d+1}})_\p^i$, hence $(\widetilde{x_1,\ldots,x_{d+1}})^i_\p \subseteq (\widetilde{f})_\p$. As $f$ has degree $m>i$ in $S$, we have $f\in (x_1,\ldots,x_{d+1})^i$, hence $(\widetilde{f})_\p= (\widetilde{x_1,\ldots,x_{d+1}})^i_\p$. However, $(\widetilde{f})_\p= (\widetilde{x_1,\ldots,x_{d+1}})^m_\p$ as well, which we had seen in the proof of \Cref{DandA}. Thus we have $(\widetilde{x_1,\ldots,x_{d+1}})^i_\p = (\widetilde{x_1,\ldots,x_{d+1}})^m_\p$ in $\rr(J)_\p$. Now contracting back to $\rr(J)$ and noting that the powers and symbolic powers of this ideal agree by \Cref{colons}, we have
$$ (\widetilde{x_1,\ldots,x_{d+1}})^i= (\widetilde{x_1,\ldots,x_{d+1}})^{(i)} = (\widetilde{x_1,\ldots,x_{d+1}})^{(m)} =(\widetilde{x_1,\ldots,x_{d+1}})^m$$
which is impossible.
\end{proof}

\section{gcd-iterations}\label{iterationssec}

In this section we present a recursive algorithm which produces equations in $\A$. This algorithm is an adaptation of the method of \textit{modified Jacobian dual iterations} used in \cite{Weaver} and is similar to the methods used in \cite{CHW} and \cite{BM}. This process consists of matrix constructions analogous to a modified Jacobian dual. We begin this section by studying how the maximal minors of such matrices factor, in a manner similar to the proof of \Cref{JDminors}.

\begin{proposition}\label{gcds}
With the assumptions of \Cref{setting1} and $C$ any column with $d+1$ entries in $S[T_1,\ldots,T_{d+1}]$, consider the $(d+1) \times (d+2)$ matrix $\mathfrak{B}=[B(\psi)\,|\,C]$ and let $\mathfrak{B}_j$ denote the submatrix obtained by deleting the $j^{\text{th}}$ column of $\mathfrak{B}$. There exists a polynomial $\mathfrak{g} \in S[T_1,\ldots, T_{d+1}]$ such that for all $1\leq j \leq d+1$, one has $\det \mathfrak{B}_j = (-1)^{j+1} T_j \cdot \mathfrak{g}$. In particular, $\mathfrak{g}$ is the greatest common divisor of the maximal minors of $\mathfrak{B}$.
\end{proposition}

\begin{proof}
We modify the proof of \Cref{JDminors}. Letting $\x = x_1,\ldots,x_{d+1}$ and $\T = T_1,\ldots,T_{d+1}$, notice that
$$[\,\x\,] \cdot B(\psi) \cdot [\,\T\,]^t = [\,\T\,] \cdot \psi \cdot [\,\T\,]^t =0 $$
as $\psi$ is an alternating matrix. As $B(\psi)$ consists of entries in $k[T_1,\ldots,T_{d+1}]$, it follows that $B(\psi) \cdot [\,\T\,]^t =0$. Let $[\,\T\,|\,0\,]$ denote the row vector $[T_1 \ldots T_{d+1}\, 0]$ and notice that $\mathfrak{B}\cdot [\,\T\,|\,0\,]^t =0$. Now applying \cref{crlemma} to $[\,\T\,|\,0\,]$ and the transpose of $\mathfrak{B}$, we see that
$$T_i \cdot (\det \mathfrak{B}_j) = (-1)^{i-j} \,T_j \cdot (\det \mathfrak{B}_i)$$
in $S[T_1,\ldots,T_{d+1}]$ for all $1\leq i,j\leq d+1$, and the claim follows.
\end{proof}

\begin{remark}\label{detBpsizero}
Notice that $\mathfrak{B}$ has $d+2$ columns, yet we purposely omit the index $j=d+2$ in \Cref{gcds}. Applying \Cref{crlemma} in the proof above at this index shows only that $\det B(\psi)=0$. However, this can already be seen using Cramer's rule as $B(\psi) \cdot [\,\T\,]^t =0$, or by noting that $J$ is of linear type by \Cref{Jlineartype}.
\end{remark}

With \Cref{gcds}, we may now introduce the method of gcd-iterations. Once again, we adopt the bigrading on $S[T_1,\ldots,T_{d+1}]$ given by $\deg x_i = (1,0)$ and $\deg T_i = (0,1)$ throughout this section.

\begin{algorithm}\label{gcdit}
We recursively define pairs consisting of a matrix and an ideal. Set $\BB_1= B$ and $\BL_1 =\LL$ for $B$ a modified Jacobian dual and $\LL$ as in \Cref{notation1}. Assume that $2\leq i\leq m$ and the following pairs $(\BB_1,\BL_1), \ldots, (\BB_{i-1},\BL_{i-1})$ have been constructed inductively. To construct the $i^{\text{th}}$ pair $(\BB_i,\BL_i)$, let $g_{i-1} = \gcd I_{d+1}(\BB_{i-1})$ and set
 $$ \BL_i = \BL_{i-1} + (g_{i-1}),\quad \quad \BB_i =[B(\psi)\,\vert\,\partial g_{i-1}]$$
 where $\partial g_{i-1}$ is a column consisting of bihomogeneous entries with constant bidegree such that
$$g_{i-1} = [x_1\ldots x_{d+1}]\cdot \partial g_{i-1}$$ 
as in \Cref{delnotation}. We refer to the pair $(\BB_{i},\BL_{i})$ as the $i^{\text{th}}$ \textit{gcd-iteration} of $(B,\LL)$.
\end{algorithm}

Notice that these matrices resemble a modified Jacobian dual, hence it is understood how these greatest common divisors arise by \Cref{gcds}. Recall from \Cref{delnotation} that, as a convention, if $g_{i-1}=0$ then $\partial g_{i-1}$ consists of zeros. Thus the next equation $g_i$, and every other proceeding equation, vanishes as well. Eventually it will be shown that these equations are nonzero, but for now we retain this possibility.

 \begin{proposition}\label{deggcdprop}
In the setting of \Cref{gcdit}, if $\gcd I_{d+1}(\BB_i) \neq 0$ for some $1\leq i\leq m$, then it has bidegree $\deg (\gcd I_{d+1}(\BB_i)) = (m-i,i(d-1))$. 
\end{proposition}

\begin{proof}
As these equations are defined recursively, we proceed by induction. In the case $i=1$, notice that $\BB_1 = [B(\psi)\,\vert\,\partial f]$, a modified Jacobian dual matrix. Noting that $B(\psi)$ consists of linear entries in $k[T_1,\ldots,T_{d+1}]$ and $\partial f$ consists of entries in $S[T_1,\ldots, T_{d+1}]$ of bidegree $(m-1,0)$, the initial claim follows from \Cref{gcds}. Now suppose that $i\geq 2$ and $\deg (\gcd I_{d+1}(\BB_j)) = (m-j,j(d-1))$ for all $1\leq j \leq i-1$, if these equations are nonzero.

Notice that if $g_i=\gcd I_{d+1}(\BB_i)$ is nonzero, then $g_{i-1} = \gcd I_{d+1}(\BB_{i-1})$ is nonzero as well, hence $\deg g_{i-1} = (m-i+1,(i-1)(d-1))$ by the induction hypothesis. Thus the entries of $\partial g_{i-1}$ are bihomogeneous with bidegree $(m-i,(i-1)(d-1))$. Again noting that the entries of $B(\psi)$ are of bidegree $(0,1)$,  it follows from \Cref{gcds} that $\deg (\gcd I_{d+1}(\BB_i)) = (m-i,i(d-1))$.
\end{proof}

Notice that the method of gcd-iterations terminates after $m$ steps in \Cref{gcdit}. If $g_m=\gcd I_{d+1}(\BB_m)$ is nonzero, then it must have bidegree $(0,m(d-1))$ following \Cref{deggcdprop}. Thus there is no column corresponding to $g_m$, as in \Cref{delnotation}, and so the process must terminate. If $g_m=0$, the process could continue, however every subsequent equation is zero as well. Thus the same ideal is achieved by the $m^{\text{th}}$ step regardless.

Following \Cref{delnotation}, the matrices in \Cref{gcdit} are not unique as there are often multiple choices for the last column. Regardless, we claim that the ideals produced by this algorithm are well-defined. First however, we provide a short lemma which will be used frequently.

\begin{lemma}[{\cite[4.4]{BM}}]\label{WD}
Let $R$ be a ring and $\a=a_1,\ldots,a_r$ an $R$-regular sequence. If $B$ and $B'$ are two matrices with $r$ rows satisfying $(\a\cdot B) = (\a\cdot B')$, then $(\a\cdot B) +I_r(B) = (\a\cdot B') +I_r(B')$.
\end{lemma}

With this, we now show that the ideals obtained from \Cref{gcdit} are well-defined.

\begin{proposition}\label{gcdwd}
The ideals $\BL_i$ and  $\BL_i+(\gcd I_{d+1}(\BB_i))$ are well-defined for $1\leq i \leq m$. 
\end{proposition}

\begin{proof}
We proceed by induction. For $i=1$, note that $\BL_1 = \LL$ is certainly well-defined as it is the ideal defining $\sym(I)$, as a quotient of $S[T_1,\ldots,T_{d+1}]$. Now suppose that $B$ and $B'$ are two candidates for $\BB_1$. In other words, $B$ and $B'$ are two modified Jacobian dual matrices. Write $B = [B(\psi)\,|\,C]$ and $B' = [B(\psi)\,|\,C']$ where $C$ and $C'$ are columns with $[\,\x\,]\cdot C = f= [\,\x\,]\cdot C'$, as in \Cref{delnotation} and \Cref{mjddefn}, where $\x = x_1,\ldots,x_{d+1}$. By \Cref{gcds}, there exist polynomials $g$ and $g'$ in $S[T_1,\ldots,T_{d+1}]$ such that $\det B_j= (-1)^{j+1}T_j g$ and $\det B_j'= (-1)^{j+1}T_j g'$. Here $B_j$ and $B_j'$ denote the submatrices of $B$ and $B'$, respectively,  obtained by deleting the $j^{\text{th}}$ column, for $1\leq j\leq d+1$. We must show that $\LL+(g) = \LL+(g')$ to complete the initial step. There is nothing to be shown if both $g$ and $g'$ are zero, hence we may assume that $g\neq 0$, without loss of generality.

Deleting the first columns of $B$ and $B'$, by \Cref{WD} we have $(\ell_2,\ldots,\ell_{d+1},f) + \det(B_1) = (\ell_2,\ldots,\ell_{d+1},f) + \det(B_1')$. Thus from \Cref{gcds}, we have
\begin{equation}\label{g1eqn1}
(\ell_2,\ldots,\ell_{d+1},f) + (gT_1)= (\ell_2,\ldots,\ell_{d+1},f) + (g'T_1).
\end{equation}
With this, we see that $gT_1 \in (\ell_2,\ldots,\ell_{d+1},f) + (g'T_1)$. However, recall that $\deg f = (m,0)$ and $\deg gT_1 = (m-1,d)$  by \Cref{deggcdprop}. Hence it follows that $gT_1 \in (\ell_2,\ldots,\ell_{d+1}) + (g'T_1)$. If $g'\neq 0$, repeating this argument shows that $g'T_1 \in (\ell_2,\ldots,\ell_{d+1}) + (gT_1)$ as well. If $g'=0$, this inclusion clearly still holds. With this, (\ref{g1eqn1}) can be refined as
\begin{equation}\label{g1eqn2}
(\ell_2,\ldots,\ell_{d+1}) + (gT_1)= (\ell_2,\ldots,\ell_{d+1}) + (g'T_1).
\end{equation}
Hence we have 
\begin{equation}\label{g1eqn3}
(\ell_1,\ldots,\ell_{d+1}) + (gT_1)= (\ell_1,\ldots,\ell_{d+1}) + (g'T_1).
\end{equation}

Recall that $\H =(\ell_1,\ldots,\ell_{d+1})$, as in \Cref{notation2}, is a prime ideal by \Cref{PropertiesOfA}. Since $T_1 \notin \H$, it follows that $(\ell_1,\ldots,\ell_{d+1})+(g) = (\ell_1,\ldots,\ell_{d+1})+(g')$, hence $\LL+(g) = \LL+(g')$ as required.

We are finished if $m=1$, so assume that $m\geq 2$. For the inductive step, assume that both $\BL_j$ and $\BL_j+(\gcd I_{d+1}(\BB_j))$ are well-defined for all $1\leq j\leq i-1 $ for some $i\leq m$. Prior to the $i^{\text{th}}$ step in \Cref{gcdit}, suppose that $B_{i-1}$ and $B_{i-1}'$ are two gcd-iteration matrices. From the induction hypothesis, we have $\BL_{i-1}+(\gcd I_{d+1}(B_{i-1})) = \BL_{i-1}+(\gcd I_{d+1}(B_{i-1}'))$, which shows that $\BL_{i}$ is well-defined. In the $i^{\text{th}}$ iteration, suppose that $B_i$ and $B_i'$ are two candidates for $\BB_i$. Setting $g_{i-1}=\gcd I_{d+1}(B_{i-1})$ and $g_{i-1}'=\gcd I_{d+1}(B_{i-1}')$, we may write $B_i = [B(\psi)\,\vert\,\partial g_{i-1}]$ and $B_i' = [B(\psi)\,\vert\,\partial g_{i-1}']$, where $\partial g_{i-1}$ and $\partial g_{i-1}'$ are two columns as in \Cref{delnotation}. Writing $g_i =  \gcd I_{d+1}(B_i)$ and $g_i' =  \gcd I_{d+1}(B_i')$, we must show that $\BL_i +(g_i) = \BL_i+(g_i')$. As before, there is nothing to be shown if both $g_i$ and $g_i'$ are zero, hence we may assume that $g_i \neq 0$, without loss of generality.

Notice that as $g_i \neq 0$, we have $g_{i-1}\neq 0$ as well. With this we claim that $(\ell_1,\ldots,\ell_{d+1}) + (g_{i-1}) = (\ell_1,\ldots,\ell_{d+1}) + (g_{i-1}')$. We had already seen this for $i=2$ in the proof of the initial case. If $i\geq 3$, notice that the equality $\BL_{i-1}+(g_{i-1}) = \BL_{i-1}+(g_{i-1}')$ from the induction hypothesis shows that
\begin{equation}\label{giminus1eqn1}
    g_{i-1} \in \BL_{i-1}+(g_{i-1}') = (\ell_1,\ldots,\ell_{d+1}) +(f,g_1,\ldots,g_{i-2}) + (g_{i-1}'),
\end{equation}
where $g_1,\ldots,g_{i-2}$ are previous equations, following \Cref{gcdit}.

Since $\BL_{i-1}$ is well-defined and $g_{i-1}$ is nonzero, it follows that $g_1,\ldots,g_{i-2}$ are nonzero as well. Thus $f,g_1,\ldots,g_{i-2}$ each have bidegree with first component at least $m-i+2$, following \Cref{deggcdprop}. Moreover, we also have $\deg g_{i-1} = (m-i+1, (i-1)(d-1))$ by \Cref{deggcdprop}. By degree considerations, it then follows that (\ref{giminus1eqn1}) can be refined as
\begin{equation}\label{giminus1eqn2}
    g_{i-1} \in  (\ell_1,\ldots,\ell_{d+1}) + (g_{i-1}').
\end{equation}
A similar argument shows that $g_{i-1}' \in  (\ell_1,\ldots,\ell_{d+1}) + (g_{i-1})$ if $g_{i-1}' \neq 0$ and if $g_{i-1}' =0$, this clearly holds. Thus $(\ell_1,\ldots,\ell_{d+1}) + (g_{i-1}) = (\ell_1,\ldots,\ell_{d+1}) + (g_{i-1}')$ as claimed.

With the equality above, we may write $g_{i-1} = u\cdot g_{i-1}'+y$ for bihomogeneous elements $u\in S[T_1,\ldots,T_{d+1}]$ and $y\in (\ell_1,\ldots,\ell_{d+1})$. If $g_{i-1}' \neq 0$, by \Cref{deggcdprop} and degree considerations it follows that $u$ must be a unit. If $g_{i-1}' =0$, we may clearly assume that $u$ is a unit. With this, the column $\partial g_{i-1}$ can be rewritten as $\partial g_{i-1} = u\cdot \partial g_{i-1}' + \partial y$, where $\partial y = \partial g_{i-1} - u\cdot  \partial g_{i-1}'$. Thus the equality 
$$(\ell_2,\ldots,\ell_{d+1}, g_{i-1}) = (\ell_2,\ldots,\ell_{d+1},u\cdot g_{i-1}'+y)$$
and \Cref{WD} show that
$$(\ell_2,\ldots,\ell_{d+1}, g_{i-1}) + \det (B_i)_1 = (\ell_2,\ldots,\ell_{d+1},u\cdot g_{i-1}'+y) + \det [B(\psi)\,|\,u\cdot \partial g_{i-1}' + \partial y]_1.$$

By \Cref{gcds} and multilinearity of determinants, we then have
\begin{equation}\label{gieqn1}(\ell_2,\ldots,\ell_{d+1}, g_{i-1}) + (g_iT_1) = (\ell_2,\ldots,\ell_{d+1},u g_{i-1}'+y) + (u g_i'T_1 + y'T_1),
\end{equation}
where $y' = \gcd I_{d+1}([B(\psi)\,|\,\partial y])$, following \Cref{gcds}. However, recall that $y\in (\ell_1,\ldots,\ell_{d+1})$, hence $y'T_1= \det [B(\psi)\,|\, \partial y]_1 \in (\ell_1,\ldots,\ell_{d+1})$, by \Cref{WD} and \Cref{detBpsizero}. From (\ref{gieqn1}) we then obtain
\begin{equation}\label{gieqn2}(\ell_1,\ldots,\ell_{d+1}, g_{i-1}) + (g_iT_1) = (\ell_1,\ldots,\ell_{d+1},g_{i-1}') + (g_i'T_1),
\end{equation}
noting that $u$ is a unit.

With (\ref{gieqn2}) above, we have $g_i T_1 \in (\ell_1,\ldots,\ell_{d+1},g_{i-1}') + (g_i'T_1)$. If $g_{i-1}'\neq 0$, then it has bidegree $(m-i+1,(i-1)(d-1))$, hence $g_i T_1 \in (\ell_1,\ldots,\ell_{d+1}) + (g_i'T_1)$ as $\deg g_i = (m-i,i(d-1))$, using \Cref{deggcdprop}. If $g_{i-1}' = 0$, then $g_i'=0$, hence (\ref{gieqn2}) shows that $g_i T_1 \in (\ell_1,\ldots,\ell_{d+1}) + (g_i'T_1)$ in this case as well. A similar argument shows that $g_i' T_1 \in (\ell_1,\ldots,\ell_{d+1}) + (g_iT_1)$. Thus we obtain 
\begin{equation}\label{gieqn3}(\ell_1,\ldots,\ell_{d+1}) + (g_iT_1) = (\ell_1,\ldots,\ell_{d+1}) + (g_i'T_1).
\end{equation}
Again noting that $\H = (\ell_1,\ldots,\ell_{d+1})$ is a prime ideal and $T_1\notin \H$, it follows that $(\ell_1,\ldots,\ell_{d+1}) + (g_i) = (\ell_1,\ldots,\ell_{d+1}) + (g_i')$. Hence $\BL_i+(g_i) = \BL_i +(g_i')$, as $(\ell_1,\ldots,\ell_{d+1}) \subset \BL_i$.
\end{proof}

\begin{proposition}\label{gcdsinA}
There is a containment of ideals, $\BL_m+(\gcd I_{d+1}(\BB_m)) \subseteq \A$.
\end{proposition}

\begin{proof}
We show that $\BL_i+\gcd I_{d+1}(\BB_i) \subseteq \A$ for $1\leq i\leq m$, inductively. For $i=1$, we clearly have $\BL_1 = \LL \subset \A$. By Cramer's rule, we have $I_{d+1} (\BB_1) = I_{d+1}(B) \subseteq \LL:(x_1,\ldots,x_{d+1})\subseteq \A$. Writing $g_1= \gcd I_{d+1}(B)$, by \Cref{gcds} we have $(g_1)(T_1,\ldots,T_{d+1}) =I_{d+1}(\BB_1) \subseteq \A$. Modulo $\A$, the image of $(T_1,\ldots,T_{d+1})$ in $\rr(I)$ is an ideal of positive grade, which is annihilated by the image of $g_1$. Thus $g_1 \in \A$ and so $\BL_1+(\gcd I_{d+1}(\BB_1)) \subseteq \A$, which completes the initial step.

If $m=1$, we are finished, so assume that $m\geq 2$ and $\BL_j+ (\gcd I_{d+1}(\BB_j)) \subseteq \A$ for all $1\leq j\leq i-1$ for some $i \leq m$. By \Cref{gcdit} and the induction hypothesis, we have $\BL_i = \BL_{i-1} + (\gcd I_{d+1}(\BB_{i-1})) \subseteq \A$. Hence we must show that $\gcd I_{d+1}(\BB_i)\in \A$. Writing $g_{i-1} = \gcd I_{d+1}(\BB_{i-1})$, by Cramer's rule we see that
$$I_{d+1} (\BB_i) \subseteq (\ell_1,\ldots,\ell_{d+1},g_{i-1}):(x_1,\ldots,x_{d+1}) \subseteq \A : (x_1,\ldots,x_{d+1}) = \A$$
as $\A = \LL:(x_1,\ldots,x_{d+1})^\infty$. Using \Cref{gcds} once more, we see that $(g_i)(T_1,\ldots,T_{d+1}) =I_{d+1}(\BB_i) \subseteq \A$, where $g_i = \gcd I_{d+1}(\BB_i)$. Thus $g_i\in \A$ as the image of $g_i$ in $\rr(I)$ annihilates an ideal of positive grade, just as before.
\end{proof}


With the containment $\BL_m+(\gcd I_{d+1}(\BB_m)) \subseteq \A$, we aim to find sufficient criteria to ensure this is an equality. Recall from \Cref{DandA} that $\widetilde{\A}= \frac{\widetilde{f}\,\widetilde{\K}^{(m)}}{\widetilde{x_{d+1}}^m}$. We now provide a similar description of the ideal obtained from \Cref{gcdit}. By \Cref{colons}, we have $(\widetilde{x_1,\ldots,x_{d+1}}) \widetilde{\K} \subseteq (\widetilde{x_{d+1}})$, hence $(\widetilde{x_1,\ldots,x_{d+1}})^i \widetilde{\K}^i \subseteq (\widetilde{x_{d+1}}^i)$. Thus we may consider the $\rr(J)$-ideal $\frac{\widetilde{f} \,\widetilde{\K}^i}{\widetilde{x_{d+1}}^i}$ for any $1\leq i\leq m$.

\begin{theorem}\label{equal}
With the assumptions of \Cref{setting1} and $\K$ as in \Cref{notation2}, one has $\frac{\widetilde{f} \,\widetilde{\K}^m}{\widetilde{x_{d+1}}^m} = ({ \BL_{m}+(\gcd I_{d+1}(\BB_m))})^{\sim}$ in $\rr(J)$. 
\end{theorem}

\begin{proof}
We proceed in a manner similar to the proof of \cite[4.7]{BM}. Letting $D_i = \frac{\widetilde{f} \,\widetilde{\K}^i}{\widetilde{x_{d+1}}^i}$ and  $D_i' =( \BL_{i}+(\gcd I_{d+1}(\BB_i)))^\sim$, it is clear that $D_i \subseteq D_{i+1}$ and $D_i' \subseteq D_{i+1}'$ for any $1\leq i\leq m-1$. We show that $D_i = D_i'$ for all $1\leq i\leq m$ by induction.

Suppose that $i=1$. We first show that $D_1\subseteq D_1' =(\LL + (\gcd I_{d+1}(B)))^\sim$. As noted in the proof of \Cref{PropertiesOfA}, $\K$ may be written as $\K=(\ell_1',\ldots,\ell_{d+1}') + (\gcd I_d(B'))+(x_{d+1})$, where $B'$ is the submatrix obtained by deleting the last row of $B(\psi)$ and $[\ell_1' \ldots \ell_{d+1}'] = [x_1 \ldots x_d] \cdot B'$. Thus modulo $\H$, we see $(\ell_1'\widetilde{,\ldots,}\ell_{d+1}') \subset (\widetilde{x_{d+1}})$ and so $\frac{\widetilde{f} \,(\ell_1'\widetilde{,\ldots,}\ell_{d+1}')}{\widetilde{x_{d+1}}}\subseteq (\widetilde{f}) = \widetilde{\LL}$. Let $g' = \gcd I_d(B')$ and recall that $B'$ is the Jacobian dual of $\psi'$, an alternating matrix as in \Cref{J'ideal}, hence the minors of $B'$ factor in a similar manner as the minors of $B$, following \Cref{JDminors} and \Cref{gcds}. Let $B_1'$ and $B_1$ denote the submatrices of $B'$ and $B$, respectively, which are obtained by deleting their first columns. By Cramer's rule, in $\rr(J)$ we have $\widetilde{\det B_1} \cdot \widetilde{x_{d+1}} =\widetilde{f}\cdot \widetilde{\det B_1'}$. By \Cref{JDminors} we have $\det B_1' = g'\cdot T_1$ and by \Cref{gcds} we have $\det B_1 = g_1\cdot T_1$, where $g_1 = \gcd I_{d+1}(B)$. Since $\rr(J)$ is a domain, it follows that $\widetilde{g_1} \cdot \widetilde{x_{d+1}} =\widetilde{f}\cdot \widetilde{g'}$. Thus $\frac{\widetilde{f} \widetilde{g'}}{\widetilde{x_{d+1}}}= \widetilde{g_1}$ and so $D_1\subseteq D_1'$.

To show the reverse containment $D_1'\subseteq D_1$, recall $x_{d+1} \in \K$ and so $\widetilde{f} = \frac{\widetilde{f}\widetilde{x_{d+1}}}{\widetilde{x_{d+1}}} \in D_1$. Thus $\widetilde{\BL_1} =\widetilde{\LL} \subset D_1$ and so it suffices to show that $\widetilde{g_1}\in D_1$, where $g_1= \gcd I_{d+1}(B)$ as before. However, from the previous argument above we have $\widetilde{g_1} = \frac{\widetilde{f} \widetilde{g'}}{\widetilde{x_{d+1}}} \in D_1$, and so $D_1'\subseteq D_1$. Thus $D_1' = D_1$, which completes the initial step.

We are finished if $m=1$, so we may assume that $m\geq 2$ and $D_j=D_j'$ for all $1\leq j \leq i-1$ for some $i\leq m$. We begin by showing that $D_i \subseteq D_i'$. Consider $\frac{\widetilde{f} \widetilde{w_1}\cdots \widetilde{w_i}}{\widetilde{x_{d+1}}^i} \in D_i$, for $w_1,\ldots, w_i \in \K$. Writing $\widetilde{w'} = \frac{\widetilde{f} \widetilde{w_1}\cdots \widetilde{w_{i-1}}}{\widetilde{x_{d+1}}^{i-1}}$, notice that $\widetilde{w'} \in D_{i-1} = D_{i-1}' =( \BL_{i-1}+(\gcd I_{d+1}(\BB_{i-1})))^\sim$ by the induction hypothesis. With this we show that $\frac{\widetilde{f} \widetilde{w_1}\cdots \widetilde{w_i}}{\widetilde{x_{d+1}}^i } = \frac{\widetilde{w'}\widetilde{w_i}}{\widetilde{x_{d+1}}}$ is contained in $D_i'$. If $w' \in \BL_{i-1}$, then $\widetilde{w'} \in D'_{i-2}= D_{i-2}$ if $i>2$, and $\widetilde{w'} \in (\widetilde{f})$ if $i=2$. In either case, we have $\frac{\widetilde{w'}\widetilde{w_i}}{\widetilde{x_{d+1}}} \in D_{i-1} = D_{i-1}' \subseteq D_i'$, by the induction hypothesis, and we are finished. Hence we may assume that $w' \in (g_{i-1})$, where $g_{i-1} = \gcd I_{d+1}(\BB_{i-1})$, and it is enough to take $w' = g_{i-1}$.

As noted in the proof of \Cref{Kscm}, recall that $\widetilde{\K} = (\widetilde{g'},\widetilde{x_{d+1}})$, where $g'=\gcd I_d(B')$. As $\widetilde{w_i}\in \widetilde{\K}$, there are two cases to consider. If $\widetilde{w_i}\in (\widetilde{x_{d+1}})$, then $\frac{\widetilde{w'}\widetilde{w_i}}{\widetilde{x_{d+1}}} \in D_{i-1}' \subseteq D_i'$ and we are finished. Thus we may assume that $w_i \in (g')$ and it is enough to take $w_i =g'$. Let $(\BB_i)_1$ denote the submatrix obtained by deleting the first column of $\BB_i$. By Cramer's rule, in $\rr(J)$ we have $\widetilde{x_{d+1}}\cdot \widetilde{\det (\BB_{i})_1} =\widetilde{g_{i-1}}\cdot \widetilde{\det (B')_1}$, where $(B')_1$ is the submatrix of $B'$ obtained by deleting its first column as before. Once again, we have $\det B_1' = g'\cdot T_1$ by \Cref{JDminors} and $\det (\BB_i)_1 = g_i\cdot T_1$ by \Cref{gcds}, where $g_i = \gcd I_{d+1} (\BB_i)$. Again noting that $\rr(J)$ is a domain, it then follows that $\widetilde{x_{d+1}} \cdot \widetilde{g_i} =\widetilde{g_{i-1}}\cdot \widetilde{g'}$. Thus $\frac{\widetilde{w'}\widetilde{w_i}}{\widetilde{x_{d+1}}} = \frac{\widetilde{g_{i-1}}\cdot \widetilde{g'}}{\widetilde{x_{d+1}}} = \widetilde{g_i} \in D_i'$, which shows that $D_i \subseteq D_i'$.

To show that $D_i' \subseteq D_i$, note that $\widetilde{\BL_i} = (\BL_{i-1} + (\gcd I_{d+1}(\BB_{i-1})))^\sim$ following \Cref{gcdit}. Thus $\widetilde{\BL_i} = D_{i-1}' = D_{i-1}\subset D_i$, by the induction hypothesis. Moreover, the previous argument shows that $\widetilde{g_i} = \frac{\widetilde{g_{i-1}}\cdot \widetilde{g'}}{\widetilde{x_{d+1}}}$, where $g_i = \gcd I_{d+1}(\BB_i)$, $g_{i-1} = \gcd I_{d+1}(\BB_{i-1})$, and $g' = \gcd I_d(B')$. Notice that $g_{i-1} \in D_{i-1}' = D_{i-1}$ by the induction hypothesis, and $g' \in \K$. Thus $\widetilde{g_i} = \frac{\widetilde{g_{i-1}}\cdot \widetilde{g'}}{\widetilde{x_{d+1}}} \in D_i$ which shows that $D_i' \subseteq D_i$. Hence $D_i'= D_i$, which completes the induction.
\end{proof}

We end this section by giving a necessary and sufficient condition for when the ideal of gcd-iterations coincides with $\A$.

\begin{corollary}\label{powersym}
With the assumptions of \Cref{setting1}, $\A = \BL_m + (\gcd I_{d+1}(\BB_m))$ in $S[T_1,\ldots,T_{d+1}]$ if and only if $\widetilde{\K}^m = \widetilde{\K}^{(m)}$ in $\rr(J)$.
\end{corollary}

\begin{proof}
Combining \Cref{equal}, \Cref{gcdsinA}, and \Cref{DandA}, we have
$$\frac{\widetilde{f} \,\widetilde{\K}^m}{\widetilde{x_{d+1}}^m} = (\BL_m + (\gcd I_{d+1}(\BB_m)))^\sim \subseteq \widetilde{\A} = \frac{\widetilde{f} \,\widetilde{\K}^{(m)}}{\widetilde{x_{d+1}}^m}
$$
and the claim follows, noting that $\rr(J)$ is a domain.
\end{proof}

\section{The Main Result}\label{defidealsec}

In \Cref{powersym}, a condition was given for when the algorithm of gcd-iterations yields a generating set of $\A$, the ideal defining $\rr(I)$ as a quotient of $S[T_1,\ldots,T_{d+1}]$. In this section, we show that this condition is satisfied and moreover, this iterative method produces a \textit{minimal} generating set of $\A$. Additional properties of $\rr(I)$ such as Cohen-Macaulayness and Castelnuovo-Mumford regularity are studied as well. We proceed in the same manner as section 5 of \cite{BM} throughout.

\begin{proposition}\label{hgtIdB}
In $S[T_1,\ldots, T_{d+1}]$, one has $\hgt I_{d}(B(\psi)) \geq 2$.
\end{proposition}

\begin{proof}
Recall that $J$ is of linear type by \Cref{Jlineartype}, hence $\sym(J)\cong \rr(J)$. In particular, $\sym(J)$ is a domain of dimension $d+2$, following \Cref{PropertiesOfA}. As $\psi$ consists of linear entries in $S$ and $B(\psi)$ consists of linear entries in $k[T_1,\ldots,T_{d+1}]$, there is an isomorphism of symmetric algebras $\sym(J) \cong \sym_{k[\T]}(E)$, where $E=\coker B(\psi)$. Since $\sym(J)$ is a domain, by \cite[6.8]{HSV2} we have
$$d+2 = \dim \sym(J) = \dim {\rm Sym}_{k[\T]}(E) = \rk E + \dim k[T_1,\ldots,T_{d+1}].$$
Thus $\rk E =1$, hence by \cite[6.8]{HSV2} we have $\hgt I_d(B(\psi)) \geq 2$.
\end{proof}

We now present the main result of this paper.

\begin{theorem}\label{mainresult}
With the assumptions of \Cref{setting1}, we have 
$$\A = \BL_m + (\gcd I_{d+1}(\BB_m)).$$
Moreover, the defining ideal $\J$ of $\rr(I)$ satisfies 
$$\J =\overline{\BL_m + (\gcd I_{d+1}(\BB_m))}$$
where $\overline{\,\cdot\,}$ denotes images modulo $(f)$. 
\end{theorem}

\begin{proof}
We proceed as in the proof of \cite[5.3]{BM}. By \Cref{powersym}, it suffices to show that $\widetilde{\K}^m = \widetilde{\K}^{(m)}$ in order to conclude that $\A = \BL_m + (\gcd I_{d+1}(\BB_m))$, from which it follows that $\J=\overline{\BL_m + (\gcd I_{d+1}(\BB_m))}$. Recall that $\widetilde{\K}$ is a strongly Cohen-Macaulay $\rr(J)$-ideal of height one and is generically a complete intersection, by \Cref{PropertiesOfA} and \Cref{Kscm}. Thus by \cite[3.4]{SV}, it is enough to show that
$$\mu(\widetilde{\K}_\p) \leq \hgt \p -1 =1$$
for any prime $\rr(J)$-ideal $\p$ containing $\widetilde{\K}$ with $\hgt \p =2$, in order to conclude that $\widetilde{\K}^m = \widetilde{\K}^{(m)}$. Let $\p$ be such a prime ideal in $\rr(J)$ and we consider two cases.

 Recall that $(\widetilde{x_1,\ldots,x_{d+1}})$ is a prime $\rr(J)$-ideal with height one by \Cref{PropertiesOfA}. If $\p \nsupseteq (\widetilde{x_1,\ldots,x_{d+1}})$, repeating the argument within the proof of \Cref{Kscm} shows that $\widetilde{\K}_\p = (\widetilde{x_{d+1}})_\p$. Thus the claim is satisfied in this case.

 Now assume that $\p\supseteq (\widetilde{x_1,\ldots,x_{d+1}})$. Recall that $\hgt I_d(B(\psi)) \geq 2$ by \Cref{hgtIdB}, and $B(\psi)$ has entries in $k[T_1,\ldots,T_{d+1}]$. Thus the ideal $(x_1,\ldots,x_{d+1})+I_d(B(\psi))$ has height at least $d+3$ in $S[T_1,\ldots,T_{d+1}]$, hence the image of this ideal in $\rr(J)$ has height at least $3$. With this, it follows that $\p \nsupseteq \widetilde{I_d(B(\psi))}$ as $\hgt \p =2$. Thus there is some $d\times d$ minor $w$ of $B(\psi)$ with $\widetilde{w}\notin \p$.

 As $B(\psi)$ is a $(d+1) \times (d+1)$ matrix, this minor $w$ is obtained by deleting row $i$ and column $j$ of $B(\psi)$, for some $1\leq i,j\leq d+1$. Deleting column $j$ of $B(\psi)$ and applying \Cref{crlemma}, we have
 \begin{equation}\label{wequation}
 \widetilde{x_{d+1}} \cdot \widetilde{w} = (-1)^{i-d-1}\widetilde{x_i}\cdot \widetilde{\det (B')_j},
 \end{equation}
 where $(B')_j$ is the submatrix of $B'$ obtained by deleting column $j$. Recall that $B' = B(\psi')$, where $\psi'$ is as in \Cref{J'ideal}. Hence by \Cref{JDminors} we have $\det (B')_j = (-1)^{j+1}T_j \cdot g'$, where $g' =\gcd I_d(B')$. Thus (\ref{wequation}) becomes
\begin{equation}\label{wequation2}
\widetilde{x_{d+1}} \cdot \widetilde{w} = (-1)^{i-d+j}\widetilde{x_i}\cdot\widetilde{T_j} \cdot \widetilde{g'}.
\end{equation}
Localizing at $\p$, $\widetilde{w}$ becomes a unit and (\ref{wequation2}) shows that $(\widetilde{x_{d+1}})_\p \in (\widetilde{g'})_\p$. Thus $\widetilde{\K}_\p = (\widetilde{g'},\widetilde{x_{d+1}})_\p = (\widetilde{g'})_\p$, and again the claim is satisfied. 
\end{proof}


\begin{corollary}\label{gcdsnotzero}
For all $1\leq i\leq m$, we have $\gcd I_{d+1}(\BB_i) \neq 0$. Additionally, $\ff(I) \cong k[T_1,\ldots,T_{d+1}]/(\mathfrak{f})$ where $\deg \mathfrak{f} = m(d-1)$.
\end{corollary}

\begin{proof}
Recall from \Cref{gcdit}, if $\gcd I_{d+1}(\BB_i) = 0$ for any $1\leq i\leq m$, then $\gcd I_{d+1}(\BB_j) =0$ for all $i\leq j\leq m$. Thus it suffices to show that $\gcd I_{d+1}(\BB_m) \neq 0$ to verify the first statement. By \Cref{mainresult}, we have $\A=\BL_m +(\gcd I_{d+1}(\BB_m))$ and we note that $\BL_m \subset (x_1,\ldots,x_{d+1})$, which follows from \Cref{gcdit} and \Cref{deggcdprop}. Hence $(x_1,\ldots,x_{d+1}) +\A = (x_1,\ldots,x_{d+1}) + (\gcd I_{d+1}(\BB_m))$ and so $\gcd I_{d+1}(\BB_m) \notin (x_1,\ldots,x_{d+1})$ since $\A \nsubseteq (x_1,\ldots,x_{d+1})$, as noted in the proof of \Cref{DandA}. In particular, $\gcd I_{d+1}(\BB_m)$ is nonzero, as claimed.

Again, noting that $(x_1,\ldots,x_{d+1}) +\A = (x_1,\ldots,x_{d+1}) + (\gcd I_{d+1}(\BB_m))$ and $\gcd I_{d+1}(\BB_m)\neq 0$, it follows from \Cref{deggcdprop} that $\gcd I_{d+1}(\BB_m)$ is of bidegree $(0,m(d-1))$. Hence $\gcd I_{d+1}(\BB_m)$ is the only equation of $\A$ contained in $k[T_1,\ldots,T_{d+1}]$, i.e. it is the only fiber equation. Thus modulo $(x_1,\ldots,x_{d+1}) +\A$, we see that $\ff(I)$ is indeed a hypersurface ring defined by an equation of degree $m(d-1)$.
\end{proof}

By \Cref{mainresult}, the method of gcd-iterations gives a generating set of $\A$. We now claim this is a minimal generating set. Recall that the \textit{relation type} of $I$ is the maximum degree, with respect to $T_1,\ldots,T_{d+1}$, of a minimal generator of the defining ideal $\J$ of $\rr(I)$. It is denoted by $\rt(I)$.

\begin{proposition}\label{mingens}
In the setting of \Cref{mainresult}, the generating set of $\A= \BL_m + (\gcd I_{d+1}(\BB_m))$ obtained from \Cref{gcdit} is minimal. In particular, $\mu(\A) = d+m+2$ and $\mu(\J) = d+m+1$. Moreover, the relation type of $I$ is $\rt(I) =m(d-1)$.
\end{proposition}

\begin{proof}
The generating set of $\A=\BL_m + (\gcd I_{d+1}(\BB_m))$ obtained from the method of gcd-iterations is $\{\ell_1,\ldots,\ell_{d+1},f,g_1,\ldots,g_m\}$ where $g_i = \gcd I_{d+1}(\BB_i)$, for $\BB_i$ a matrix as in \Cref{gcdit}. Recall that $\deg \ell_i = (1,1)$ for $1\leq i\leq d+1$, as $\psi$ consists of linear entries in $S$. Additionally, recall that $\deg f = (m,0)$. Moreover, by \Cref{gcdsnotzero} and \Cref{deggcdprop}, we have $\deg g_i = (m-i,i(d-1))$ for $1\leq i\leq m$. With this, we show that the generating set above is minimal by showing that none of these generators can be expressed in terms of the others.

First, recall that the ideal $\H =(\ell_1,\ldots,\ell_{d+1})$, as in \Cref{notation2}, is the ideal defining $\sym(J)$. As $[\ell_1 \ldots \ell_{d+1}]= [T_1 \ldots T_{d+1}]  \cdot \psi$ and $\psi$ minimally presents $J$, it follows that $\ell_1,\ldots,\ell_{d+1}$ minimally generate $\H$. To see that these are minimal generators of $\A$ as well, suppose not for a contradiction and, without loss of generality, assume $\ell_1$ is a non-minimal generator of $\A$. Thus $\ell_1$ can be written as a combination of the remaining generators $\ell_2,\ldots,\ell_{d+1},f,g_1,\ldots,g_m$. 

If $m\geq 2$, by degree considerations in both components of the bigrading, it then follows that $\ell_1$ can be written in terms of $\ell_2,\ldots,\ell_{d+1}$, which is a contradiction as $\ell_1,\ldots,\ell_{d+1}$ minimally generate $\H$. In the case $m=1$, by degree considerations once more, $\ell_1$ can then be expressed in terms of $\ell_2,\ldots,\ell_{d+1},f$. If $\ell_1 \in ( \ell_2,\ldots,\ell_{d+1})$, we achieve the same contradiction, hence it follows that there is some element $b$ of bidegree $(0,1)$ such that $b\cdot f \in \H =(\ell_1,\ldots,\ell_{d+1})$. However, recall from \Cref{PropertiesOfA} that $\H$ is a prime ideal, hence either $b\in \H$ or $f\in \H$, both of which are impossible by degree considerations.

A similar argument shows that $f$ and $g_m$ are minimal generators, as they have bidegrees $(m,0)$ and $(0,m(d-1))$ respectively. 
Now suppose that $g_i$ is a non-minimal generator of $\A$ for some $i\leq m-1$, in which case there are higher-order iterations. Writing $g_i$ as a combination of the remaining generators, by degree considerations in both components of the bigrading, it then follows that $g_i\in (\ell_1,\ldots,\ell_{d+1})$. Thus the column $\partial g_i$ in \Cref{gcdit} can be taken as a combination of the columns of $B(\psi)$ by \Cref{gcdwd}. Hence $I_{d+1}(\BB_{i+1}) = I_{d+1}(B(\psi)) = 0$ by \Cref{detBpsizero}, and so $g_{i+1} = 0$ which is a contradiction by \Cref{gcdsnotzero}.

The claim regarding $\J$ then follows as $f$ is a minimal generator of $\A$. Lastly, the relation type of $I$ is seen to be $\rt(I)=m(d-1)$, as $\deg g_m = (0,m(d-1))$.
\end{proof}

We now provide an example that illustrates the process of gcd-iterations, in order to provide a minimal generating set of $\A$. We remark that it is quite simple to perform this algorithm in a computer algebra system, such as \textit{Macaulay2} \cite{M2}.

\begin{example}
For $k$ an infinite field, let $S= k[x_1,x_2,x_3,x_4,x_5]$, $f=x_5^3$, and $R=S/(f)$. Consider the matrix $\varphi$, with entries in $R$, and its counterpart $\psi$ with entries in $S$:
\[\varphi =\begin{bmatrix}
0&\overline{x_1} & \overline{x_2} & 0 & \overline{x_4} \\[0.5ex]
-\overline{x_1}&0&\overline{x_4}& 0& \overline{x_3} \\[0.5ex]
-\overline{x_2}&-\overline{x_4}&0&\overline{x_1}& \overline{x_5}\\[0.5ex]
0&0&-\overline{x_1}&0&\overline{x_2}\\[0.5ex]
-\overline{x_4}&-\overline{x_3}&-\overline{x_5}&-\overline{x_2}&0
\end{bmatrix}
\hspace{10mm}
\psi =\begin{bmatrix}
0&x_1 & x_2 & 0 & x_4 \\[0.5ex]
-x_1&0&x_4& 0& x_3 \\[0.5ex]
-x_2&-x_4&0&x_1& x_5\\[0.5ex]
0&0&-x_1&0&x_2\\[0.5ex]
-x_4&-x_3&-x_5&-x_2&0
\end{bmatrix}
\]
where $\overline{\,\cdot\,}$ denotes images modulo $(f)$. A simple computation shows that $\hgt \pf_4(\varphi) \geq 3$, hence $I=\pf_4(\varphi)$ is a perfect Gorenstein $R$-ideal of grade $3$ by \cite[2.1]{BE}. Moreover, $I$ satisfies $G_4$ since $\hgt I_2(\varphi) = 4$, $\hgt I_3(\varphi) =3$, and $\hgt I_4(\varphi) = 3\geq 2$, which can be checked easily. Thus the assumptions of \Cref{mainresult} are met.

The Jacobian dual of $\psi$ which consists of entries in $k[T_1,T_2,T_3,T_4,T_5]$ is
\[
B(\psi) = \begin{bmatrix}
-T_2 & T_1 & -T_4 & T_3 & 0 \\[0.5ex]
-T_3 & 0 & T_1 & -T_5 & T_4\\[0.5ex]
0 & -T_5 & 0 & 0 & T_2\\[0.5ex]
-T_5 & -T_3 & T_2 & 0 & T_1\\[0.5ex]
0 & 0 & - T_5 & 0 & T_3
\end{bmatrix}
\]
hence we may construct the modified Jacobian dual and perform the method of gcd-iterations. For the purposes of notation, let $W=T_1T_3T_5 -T_2T_3^2 -T_2^2T_5 - T_4T_5^2$, which happens to be $W = \gcd I_4(B')$, following \Cref{notation2}. Following \Cref{gcdit}, we obtain
\[
\begin{array}{lll}
  \BB_1 =  \begin{bmatrix}
-T_2 & T_1 & -T_4 & T_3 & 0 &0 \\[0.5ex]
-T_3 & 0 & T_1 & -T_5 & T_4 &0\\[0.5ex]
0 & -T_5 & 0 & 0 & T_2 & 0\\[0.5ex]
-T_5 & -T_3 & T_2 & 0 & T_1 &0\\[0.5ex]
0 & 0 & - T_5 & 0 & T_3 &x_5^2
\end{bmatrix} & \quad & g_1 = \gcd I_5(\BB_1)= x_5^2 W\\
\, \\
\BB_2 =  \begin{bmatrix}
-T_2 & T_1 & -T_4 & T_3 & 0 &0 \\[0.5ex]
-T_3 & 0 & T_1 & -T_5 & T_4 &0\\[0.5ex]
0 & -T_5 & 0 & 0 & T_2 & 0\\[0.5ex]
-T_5 & -T_3 & T_2 & 0 & T_1 &0\\[0.5ex]
0 & 0 & - T_5 & 0 & T_3 &x_5 W
\end{bmatrix} & \quad & g_2 = \gcd I_5(\BB_2)= x_5 W^2\\
\, \\
\BB_3 =  \begin{bmatrix}
-T_2 & T_1 & -T_4 & T_3 & 0 &0 \\[0.5ex]
-T_3 & 0 & T_1 & -T_5 & T_4 &0\\[0.5ex]
0 & -T_5 & 0 & 0 & T_2 & 0\\[0.5ex]
-T_5 & -T_3 & T_2 & 0 & T_1 &0\\[0.5ex]
0 & 0 & - T_5 & 0 & T_3 & W^2
\end{bmatrix} & \quad & g_3 = \gcd I_5(\BB_2)= W^3,
\end{array}
\]
where the greatest common divisors of the minors occur as in \Cref{gcds}.

By \Cref{mainresult}, we have $\A= \LL + (g_1,g_2,g_3)$ and the defining ideal of $\rr(I)$ is $\J = \overline{\LL + (g_1,g_2,g_3)}$. Notice that $\A$ and $\J$ are not prime ideals, which is to be expected as $R$, and hence $\rr(I)$, is not a domain.
\end{example}

\begin{remark}\label{f linear case}
We note that \Cref{mainresult} recovers \Cref{Moreyresult} when $m=1$. After a change of coordinates, it can be assumed that the factored equation $f$ is one of the indeterminates, say $f=x_{d+1}$. Thus $R\cong k[x_1,\ldots,x_d] =S'$ and so $I$ and $J'$ are the same ideal with $\varphi = \psi'$, following the notation in \Cref{J'ideal}. Recall that the submatrix $B'$ of $B(\psi)$, as in \Cref{notation2}, is the Jacobian dual $B'=B(\psi')$, with respect to $x_1,\ldots,x_d$. Moreover, the modified Jacobian dual $B=[B(\psi)\,|\,\partial f]$ is unique in this case and the column $\partial f$ consists of all zeros except for a 1 in the last entry. Thus the greatest common divisor of the minors of $B$, the first and only gcd-iteration, is exactly the greatest common divisor of the minors of $B' = B(\psi')$.
\end{remark}

\subsection{Depth and Cohen-Macaulayness}\label{depthsec}

In the setting of \Cref{mainresult}, we now study the Cohen-Macaulay property of the Rees ring $\rr(I)$, using the isomorphism $\rr(I) \cong S[T_1,\ldots,T_{d+1}]/\A$. We continue to follow section 5 of \cite{BM} and begin by creating a handful of short exact sequences which will be useful not only to study the depth of $\rr(I)$, but also its regularity with respect to various gradings.

Let $\m =(x_1,\ldots,x_{d+1})$ and recall that $\widetilde{\K} = (\widetilde{g'},\widetilde{x_{d+1}})$, where $g'= \gcd I_d(B')$. Additionally, recall that $\widetilde{\K}$ is a strongly Cohen-Macaulay ideal by \Cref{Kscm} and $\m \rr(J) = (\widetilde{x_{d+1}}):\widetilde{\K}$ by \Cref{colons}. Thus there is a short exact sequence of bigraded $\rr(J)$-modules
\[
0\longrightarrow \m \rr(J) (0,-(d-1)) \longrightarrow \rr(J)(-1,0)\oplus \rr(J)(0,-(d-1))\longrightarrow \widetilde{\K} \longrightarrow 0.
\]

Consider the induced sequence obtained by applying the functor $\sym(-)$ to the sequence above. Taking the $m^{\text{th}}$ graded strand, we obtain
$$\hspace{10mm}\m \rr(J) (0,-(d-1)) \otimes \text{Sym}_{m-1}\big(\rr(J)(-1,0)\oplus \rr(J)(0,-(d-1))\big)\overset{\sigma}{\longrightarrow}\hspace{15mm}$$ 
$$\hspace{20mm}\text{Sym}_{m}\big(\rr(J)(-1,0)\oplus \rr(J)(0,-(d-1))\big) \longrightarrow \text{Sym}_m(\widetilde{\K}) \longrightarrow 0.$$
Notice that $\ker \sigma$ is torsion due to rank considerations. However, it is a submodule of a torsion-free $\rr(J)$-module, hence it must vanish. Thus $\sigma$ is injective and we have the short exact sequence
$$0\longrightarrow \m \rr(J) (0,-(d-1)) \otimes \text{Sym}_{m-1}\big(\rr(J)(-1,0)\oplus \rr(J)(0,-(d-1))\big)\overset{\sigma}{\longrightarrow}\hspace{15mm}$$ 
$$\hspace{20mm}\text{Sym}_{m}\big(\rr(J)(-1,0)\oplus \rr(J)(0,-(d-1))\big) \longrightarrow \text{Sym}_m(\widetilde{\K}) \longrightarrow 0.$$
As $\widetilde{\K} = (\widetilde{g'},\widetilde{x_{d+1}})$, the proof of \Cref{Kscm} shows that $\widetilde{\K}$ satisfies $G_\infty$. As $\widetilde{\K}$ is strongly Cohen-Macaulay, by \cite[2.6]{HSV1} it is an ideal of linear type, hence $\sym_m(\widetilde{\K}) \cong \widetilde{\K}^m$. Thus the short exact sequence above can be read as
\begin{equation}\label{directsumseq}
0\rightarrow \displaystyle{\bigoplus_{i=0}^{m-1}}\,\m \rr(J)\big(-i,-(m-i)(d-1)\big)\rightarrow \displaystyle{\bigoplus_{i=0}^{m}}\, \rr(J)\big(-i,-(m-i)(d-1)\big) \rightarrow \widetilde{\K}^m\rightarrow 0.
\end{equation}

We are now ready to compute the depths of $\rr(I)$, $\ff(I)$, and $\gr(I)$. Recall that a Noetherian local ring $A$ is said to be \textit{almost} Cohen-Macaulay if $\dep A \geq \dim A-1$.

\begin{theorem}\label{depth}
In the setting of \Cref{mainresult}, 
\begin{enumerate}[(a)]
\setlength\itemsep{1em}
    \item The Rees algebra $\rr(I)$ is almost Cohen-Macaulay. Moreover, $\rr(I)$ is Cohen-Macaulay if and only if $m=1$.
    
    \item The associated graded ring $\gr(I)$ is almost Cohen-Macaulay. Moreover, $\gr(I)$ is Cohen-Macaulay if and only if $m=1$.
    
    \item The special fiber ring $\ff(I)$ is Cohen-Macaulay.
\end{enumerate}
\end{theorem}

\begin{proof}
Consider the short exact sequence
\begin{equation}\label{regseq1}
0 \longrightarrow\m \rr(J) \longrightarrow \rr(J) \longrightarrow \ff(J)\longrightarrow 0
\end{equation}
and recall that $\rr(J)$ is a Cohen-Macaulay domain of dimension $d+2$ by \Cref{PropertiesOfA}. Moreover, recall that $J$ is of linear type by \Cref{Jlineartype}, hence $\ff(J) \cong k[T_1,\ldots,T_{d+1}]$. Comparing the depths of the $\rr(J)$-modules in (\ref{regseq1}), it follows that $\dep \m \rr(J) \geq d+2$ and so $\dep \m \rr(J) =d+2$, as this is the maximum possible depth. This together with (\ref{directsumseq}) shows that $\dep \widetilde{\K}^m \geq d+1$. 

We also have the short exact sequence
\begin{equation}\label{regseq2}
0 \longrightarrow \widetilde{\A} \longrightarrow \rr(J) \longrightarrow  \rr(I) \longrightarrow 0
\end{equation}
and we note that $\widetilde{\A} = \frac{\widetilde{f}\widetilde{\K}^{(m)}}{\widetilde{x_{d+1}}^n} =\frac{\widetilde{f}\widetilde{\K}^{m}}{\widetilde{x_{d+1}}^n} \cong \widetilde{\K}^m$ by \Cref{DandA}, \Cref{mainresult}, and \Cref{powersym}. Comparing the depths of the modules in (\ref{regseq2}), it follows that $\dep \rr(I) \geq d$, hence $\rr(I)$ is almost Cohen-Macaulay. The Cohen-Macaulayness in the case $m=1$ follows from \Cref{f linear case} and \Cref{Moreyresult}. If $m\geq 2$, then $\rr(I)$ is not Cohen-Macaulay by \cite[3.1]{PU}. Thus $\dep \rr(I) =d$ in this case, which shows part (a).

For the proceeding part, we note that $\gr(I)$ is certainly almost Cohen-Macaulay if it is Cohen-Macaulay. In the case that $\gr(I)$ is not Cohen-Macaulay, we have $\dep \gr(I) \geq \dep \rr(I) -1 \geq d-1$ by \cite[3.12]{HM} and (a). Thus $\gr(I)$ is almost Cohen-Macaulay. The last assertion of (b) now follows from (a) and \cite[3.1]{PU}.

The assertion on the Cohen-Macaulayness of $\ff(I)$ in (c) is clear since it is a hypersurface ring by \Cref{gcdsnotzero}.
\end{proof}

\subsection{Regularity}

We now consider the the Castelnuovo-Mumford regularity of $\rr(I)$ in the the setting of \Cref{mainresult}. We follow all definitions and conventions as given in \cite{Trung}. Once more, we proceed as in section 5 of \cite{BM} and we note that regularity is easily compared along short exact sequences \cite{Eisenbud}.

Again we use the isomorphism $\rr(I) \cong S[T_1,\ldots,T_{d+1}]/\A$ and we note that there are multiple gradings on $\rr(I)$. We consider its regularity with respect to the gradings represented by the $S[T_1,\ldots,T_{d+1}]$-ideals $\m = (x_1,\ldots,x_{d+1})$, $\tt=(T_1,\ldots,T_{d+1})$, and $\n=(x_1,\ldots,x_{d+1},T_1,\ldots,T_{d+1})$. When computing regularity with respect to $\m$, we set $\deg x_i=1$ and $\deg T_i =0$. Similarly, when computing regularity with respect to $\tt$, we set $\deg x_i=0$ and $\deg T_i =1$. Lastly, when computing regularity with respect to $\n$, we adopt the total grading and set $\deg x_i=1$ and $\deg T_i =1$.

\begin{theorem}\label{rtandreg}
In the setting of \Cref{mainresult}, we have
\[
\begin{array}{ccc}
     {\rm reg}_\tt \rr(I)  = m(d-1)-1, &  {\rm reg}_\m \rr(I)  = m-1, & {\rm reg}_\n \rr(I) \leq (m+1)(d-1)-1.  
\end{array}
\]
Additionally, $\reg \ff(I) = m(d-1)-1$.
\end{theorem}

\begin{proof}
For the regularity of $\rr(I)$ with respect to $\tt$, it is well-known that $\rt(I) -1 \leq \reg_\tt \rr(I)$ \cite[pp. 2813-2814]{Trung}. Hence $\reg_\tt \rr(I)  \geq m(d-1)-1$, using \Cref{mingens}. Similarly, it follows that $\reg_\m \rr(I) \geq m-1$, by comparing the degrees, with respect to $\m$, of the minimal generators of $\A$ in \Cref{mingens}. Thus it must be shown that $\reg_\tt \rr(I) \leq m(d-1)-1$ and $\reg_\m \rr(I) \leq m-1$, which we show with the remaining inequality simultaneously. We use the sequences (\ref{regseq1}) and (\ref{regseq2}) once more.

Recall that $J$ is of linear type by \Cref{Jlineartype}, hence $\rr(J)\cong \sym(J)$ is a domain defined by $\H = (\ell_1,\ldots,\ell_{d+1})$. Furthermore, recall that $ [\ell_1\ldots\ell_{d+1}]=[T_1\ldots T_{d+1}]\cdot \psi$ and $\psi$ is an alternating matrix. Hence a graded minimal free resolution of $\rr(J)$ is known from section 6 of \cite{Kustin}. Moreover, the resolution given in \cite[6.3]{Kustin} is amenable to each of the gradings, from which it follows that
\[
\begin{array}{ccc}
     {\rm reg}_\tt \rr(J)  =0, &  {\rm reg}_\m \rr(J)  =0, & {\rm reg}_\n \rr(J) =d-1.  
\end{array}
\]
As $J$ is of linear type, its special fiber ring is $\ff(J)  \cong k[T_1,\ldots,T_{d+1}]$. Thus
\[
\begin{array}{ccc}
     {\rm reg}_\tt \ff(J)  =0, &  {\rm reg}_\m \ff(J)  =0, & {\rm reg}_\n \ff(J) =0.  
\end{array}
\]
With this and (\ref{regseq1}), we then have 
\[
\begin{array}{ccc}
     {\rm reg}_\tt\, \m \rr(J)  \leq 1, &  {\rm reg}_\m \,\m \rr(J)  \leq 1, & {\rm reg}_\n\, \m \rr(J)  =d-1.\\
\end{array}
\]

For convenience, write
\[
\begin{array}{ccc}
    M= \displaystyle{\bigoplus_{i=0}^{m-1}\m \rr(J)\big(-i,-(m-i)(d-1)\big)},  &  N= \displaystyle{\bigoplus_{i=0}^{m} \rr(J)\big(-i,-(m-i)(d-1)\big)}   \\ 
\end{array}
\]
for the modules in (\ref{directsumseq}). With the inequalities above, it follows that
\[
\begin{array}{lcl}
   {\rm reg}_\tt M \leq m(d-1)+1,  &\quad &{\rm reg}_\tt N = m(d-1) \\[1ex]
     {\rm reg}_\m M \leq m,  & \quad &{\rm reg}_\m N = m\\[1ex]
     {\rm reg}_\n M \leq (m+1)(d-1), &\quad  &{\rm reg}_\n N = (m+1)(d-1).\\
     \end{array}
\]
Now using (\ref{directsumseq}), it then follows that
\[
\begin{array}{ccc}
     {\rm reg}_\tt \widetilde{\K}^m  \leq m(d-1), &  {\rm reg}_\m \widetilde{\K}^m  \leq m, & {\rm reg}_\n \widetilde{\K}^m \leq (m+1)(d-1).\\
\end{array}
\]
The inequalities above, the
bigraded isomorphism $\widetilde{\A} \cong \widetilde{\K}^m(0,0)$, and the sequence (\ref{regseq2}) give
\[
\begin{array}{ccc}
     {\rm reg}_\tt \rr(I)  \leq m(d-1)-1, &  {\rm reg}_\m \rr(I)  \leq m-1, & {\rm reg}_\n \rr(I) \leq (m+1)(d-1)-1.  
\end{array}
\]
as claimed. 

The assertion on the regularity of $\ff(I)$ is clear as $\ff(I)$ is a hypersurface ring defined by an equation of degree $m(d-1)$ in $k[T_1,\ldots,T_{d+1}]$, by \Cref{gcdsnotzero}. 
\end{proof}


\section*{Acknowledgements}

The author would like to thank Bernd Ulrich and Claudia Polini for many insightful conversations and discussions on the results presented here. The use of \textit{Macaulay2} \cite{M2} was of great assistance in this project.







\end{document}